\documentclass[12pt]{amsart}

\usepackage{amsfonts, amssymb, amsmath, amsthm}
\newcommand{\R}{\mathbb{R}}
\newcommand{\N}{\mathbb{N}}
\newcommand{\Q}{\mathbb{Q}}
\newcommand{\Z}{\mathbb{Z}}

\newcommand{\T}{\mathbb{T}}

\newcommand{\C}{\mathbb{C}}
\newcommand{\cK}{\mathcal{K}}

\newcommand{\cJ}{\mathcal{J}}

\newcommand{\cP}{\mathcal{P}}
\newcommand{\cS}{\mathcal{S}}

\newcommand{\Cech}{{\v Cech}{} }
\def\dl{\underrightarrow\lim}
\def\il{\underleftarrow\lim}

\newcommand{\tp}{\tilde{p}}

\newcommand{\tT}{\tilde{T}}
\newcommand{\tS}{\tilde{S}}

\newcommand{\tf}{\tilde{f}}

\newcommand{\tG}{\tilde{G}}
\newcommand{\tX}{\tilde{X}}

\newcommand{\td}{\tilde{d}}
\newcommand{\tx}{\tilde{x}}
\newcommand{\ty}{\tilde{y}}

\newcommand{\Lim}{\varprojlim} 
\newcommand{\dlim}{\varinjlim} 

\newcommand{\OP}{\Omega_{\Phi}}
\newcommand{\tO}{\tilde{\Omega}}

\newcommand{\larr}{\left( \begin{array}{c}}
\newcommand{\rarr}{\end{array} \right) }

\newcommand{\lsqarr}{\left[ \begin{array}{c}}
\newcommand{\rsqarr}{\end{array} \right]}

\newcommand{\inv}{\varprojlim}

\newtheorem{theorem}{Theorem}
\newtheorem{Theorem}[theorem]{Theorem}
\newtheorem{cor}[theorem]{Corollary}

\newtheorem{claim}[theorem]{Claim}
\newtheorem{question}[theorem]{Question}
\newtheorem{corollary}[theorem]{Corollary}
\newtheorem{conjecture}[theorem]{Conjecture}
\newtheorem{lemma}[theorem]{Lemma} \newtheorem{prop}[theorem]{Proposition}
\newtheorem{Proposition}[theorem]{Proposition}
\newtheorem{example}[theorem]{Example}
\newtheorem{remark}[theorem]{Remark}

\begin{document}
\title{Geometric realization for substitution tilings}
\author{M.~Barge and J.-M.~Gambaudo}

\date{\today}

\begin{abstract} Given an $n$-dimensional substitution $\Phi$ whose associated linear expansion $\Lambda$ is unimodular and hyperbolic, we use elements of the one-dimensional integer \v{C}ech
cohomology of the tiling space $\OP$ to construct a finite-to-one semi-conjugacy $G:\OP\to\T^D$, called geometric realization, between the substitution induced dynamics and an invariant set of a hyperbolic toral automorphism. If $\Lambda$ satisfies a Pisot family condition and the rank of the module of generalized return vectors equals the generalized degree of $\Lambda$, $G$ is surjective and coincides with the map onto the maximal equicontinuous factor of the $\R^n$-action on $\OP$. We are led to formulate a higher-dimensional generalization of the Pisot substitution conjecture: If  $\Lambda$ satisfies the Pisot family condition and the rank of the one-dimensional cohomology of $\OP$ equals the generalized degree of $\Lambda$, then the $\R^n$-action on $\OP$ has pure discrete spectrum.

\end{abstract}

\maketitle

\tableofcontents

\newpage

\section{Introduction}\label{intro}
 
Let $\Delta$ be a subset of the $n$-dimensional Euclidean space $\R^n, n\geq 1.$  A {\bf tiling} of $\Delta$ is a countable collection $T =\{t_j\}_{j\in J}$ 
of topological closed $n$-balls  in  $\Delta$, called {\bf tiles}, that cover $\Delta$ and have pairwise disjoint interiors. 
Consider  a finite collection of polyhedra  $\cP =\{\rho_1, \dots, \rho_m\}$ in $\R^n$, called {\bf prototiles}. We say that $\cP$ generates a tiling  $T =\{t_j\}_{j\in J}$ of $\Delta$ if each $t_i$ is a translated copy of one of the $\rho_j$ and 
the $t_i$'s meet full face to full face in all dimensions (so each $t_i$ has finitely many faces in each dimension, and if $t_i$ and $t_j$ meet in a point $x$ that is in the relative interior of a face of $t_i$ or $t_j$, then they meet in that entire face).  We denote by $\Omega_\cP(\Delta)$ the set of all tilings of $\Delta$ generated by $\cP$, and by $\Omega_\cP$ the set $\Omega_\cP(\R^n)$. 
When $n>1$, it is not always the case that, for a given family $\cP$, the set  $\Omega_\cP$  is not empty. In fact it is well known that the problem of whether or not $\Omega_\cP$ is empty is not decidable \cite{Berger}.  When   $\Omega_\cP\neq \emptyset$, the group $\R^n$ acts on $\Omega_\cP$ by translation: for any tiling 
$T =\{t_j\}_{j\in J}$  in $\Omega_\cP$, and for any $u\in \R^n$ we define:
$$T -u\, =\, \{t_i -u\}_{j\in J}.$$ The set $\Omega_\cP$ has a natural metric defined as follows. 
For each $r\ge0$, let $\bar{B}_r(0)$ denote the closed ball of radius $r$ centered at $0\in\R^n$ and, for $T\in\Omega_\cP$, let $B_r[T]:=\{t_j\in T:t_j\cap\bar{B}_r(0)\ne\emptyset\}$ be the collection of tiles in $T$ that meet $\bar{B}_r(0)$. Given $T,T'\in\Omega_\cP$, let $A$ denote the set of $\epsilon \in (0, 1)$ such that there are $u,u'\in \R^n$, with $|u|,|u'|<\epsilon/2$, so that $B_{1/\epsilon}[T-u]=B_{1/\epsilon}[T'-u']$. Then:
$$
d(T, T') = \left\{
    \begin{array}{ll}
        \inf A & \mbox{if }A\neq \emptyset\\
        1 & \mbox{if not.}
    \end{array}
\right.
$$
In words, $T$ and $T'$ are close if, up to small translation, they agree exactly in a large neighborhood of the origin. When $\Omega_\cP$ is equipped with this metric, the translation action is continuous. If $\Omega_\cP$ has finite local complexity (that is, there are only finitely many patterns of tiles in elements of $\Omega_\cP$ of any fixed finite radius - see below), $\Omega_\cP$ is compact and has the structure of a {\bf lamination} whose $n$-dimensional leaves are the orbits of the translation action and with totally disconnected transverse direction.

Substitutions provide an important method for constructing tilings. Consider a polyhedral family $\cP =\{\rho_1, \dots, \rho_m\}$ of prototiles in $\R^n$ and suppose there is  an expanding linear map $\Lambda:\R^n\to\R^n$ (called {\bf inflation}) so that for each $j\in \{1, \dots , m\}$, there exists a tiling $\cS_j$  of $\Lambda(\rho_j)$ generated by $\cP$. The collection of tilings $\cS= \{\cS_1, \dots, \cS_m\}$ is  called a {\bf substitution rule}. The {\bf incidence matrix} associated with $\Lambda$ and $\cS$ is the 
$m$-by-$m$ matrix $M_{\cS}=(m_{i,j})_{i,j}$, where for each $i,j\in\{1,\ldots,m\}$, 
the entry $m_{i,j}$ is the number of translated copies of $\rho_i$ in $\cS_j$.\footnote{ It is often desirable to mark, or color, prototiles and tiles so that tiles may occupy the same underlying set but
nonetheless be distinct. Formally, then, a tile is a pair $\tau=(t,m)$ where $t$ is a closed topological 
$n$-ball and $m$ is one of finitely many possible marks. We will stick to the simpler unmarked language in this paper, though all of the results are to be understood in the more general setting.}
Recall that a matrix $M$ is \emph{primitive} if there exists $k>0$ such that all the 
elements of $M^k$ are positive.

By iterating inflation followed by substitution infinitely many times, one sees that $\Omega_{\cP}\ne\emptyset$. 
On the one hand, the inflation $\Lambda$  induces a natural continuous map $\cK$
from $\Omega_{\cP}$ to $\Omega_{ \Lambda \cP}$, where $\Lambda\cP = \{\Lambda p \mid p\in\cP\}$. 
On the other hand, the substitution rule $\cS$ induces a continuous map $\cJ$ from $\Omega_{ \Lambda\cP}$ to $\Omega_{ \cP}$, 
which is defined by subdividing the tiles of a tiling in 
$\Omega_{\lambda\cP}$ according to the substitution rule. 
The composition of these two maps yields a self-map 
$\Phi:= \cJ\circ \cK : \Omega_{\cP}\to \Omega_{\cP}$ which we call the {\bf substitution map}. 
We will say that a finite collection $P=\{t_1,\ldots,t_k\}$ of tiles is an {\bf allowed patch} for $\Phi$ if 
there is a prototile $\rho_j$ and a $k\in\N$ so that $P$ is contained in some translate of the patch $\Phi^k(\rho_j)$ obtained by $k$ iterations of inflation and substitution applied to $\rho_j$.
The {\bf tiling space} associated with $\Phi$ is the collection
$$\OP:=\{T\in\Omega_{\cP}: B_r[T] \text{ is an allowed patch for } \Phi \text{ for all } r\ge0\}.$$
It is clear that $\OP$ is invariant under both translation and the substitution map.

.
Throughout this paper we will make the following three assumptions:
\begin{itemize}
\item  the map $\Phi$ is {\bf primitive}, which means that the associated substitution matrix $M_{\cS}$ is primitive;
\item  the tiling space $\OP$ has {\bf finite local complexity} (FLC), which means that for each $r>0$
there are, up to translation, only finitely many distinct patches of the form $B_r[T]:=\{t\in T:t\cap \bar{B}_r(0)\ne\emptyset\},\, T\in\OP$; and 
\item  $\Omega_{\Phi}$ is {\bf translationally non-periodic} which means that if there exist a tiling $T$ in $\Omega_\Phi$ and $u\in \R^n$ such that $T - u = T$, then $u =0$. 
\end{itemize}

Despite the fact that the substitution map is basically a linear inflation, it turns out the laminated structure of the phase space $\Omega_\phi$ induces a very rich dynamics. The action of $\R^n$ by translation on $\Omega_{\Phi}$ is minimal and uniquely ergodic (\cite{AP}); the map $\Phi$ is a homeomorphism (\cite{sol}) and is ergodic with respect to the unique $\R^n$-invariant measure $\mu$; and the dynamics on $\Omega_{\Phi}$ interact by $\Phi(T-v)=\Phi(T)-\Lambda v$.

Our first  goal in this paper is to provide a way to understand the dynamical system  $(\Omega_\Phi, \Phi)$ in terms of the standard geometric theory of dynamical systems where one studies iteration of maps on compact manifolds. 
This is what we call {\bf geometric realization}. 
More precisely, we will show that, under some assumptions on the hyperbolicity of the eigenvalues of the inflation $\Lambda$, there exists a finite-to-one continuous map from $\Omega_\Phi$ to some $D$-dimensional torus $\T^D$ that factors the dynamics of $\Phi$ into those of a linear hyperbolic map. 
 The basic idea is as follows. The one-dimensional cohomology of $\OP$ supplies a map into the $D$-torus that, on the level of cohomology, conjugates $\Phi$ with a hyperbolic  toral automorphism. The technique of global shadowing is then applied to improve the map into an actual dynamical semi-conjugacy of hyperbolic systems. (This technique originated with \cite{F} and \cite{Fr}, and is used also in \cite{BKw} to a.e. embed pseudo-Anosovs into hyperbolic toral automorphisms.)

The second goal is to establish a link between this geometric realization and the traditional Pisot Substitution Conjecture (\cite{BS}) regarding pure discreteness of the $\R^n$-action. 
On $\OP$ we have $\Phi(T-v)=\Phi(T)-\Lambda v$ and a like relation holds between the hyperbolic action on $\T^D$ and a Kronecker action on the $u$-dimensional leaves of its unstable foliation. Under an assumption on the ``Pisotness" of the inflation $\Lambda$, $u=n$ and the semi-conjugacy of hyperbolic systems becomes also a semi-conjugacy of $\R^n$-actions. The coordinate functions of the semi-conjugacy are thus eigenfunctions of $\R^n$-action and their associated eigenvalues constitute a generating set for the discrete part of the spectrum of the $\R^n$-action on $\OP$. In essence, elements of the first cohomology of $\OP$ are converted into eigenfunctions of the $\R^n$-action by means of global shadowing.

\section{Main results}
The eigenvalues of the inflation $\Lambda$ are all algebraic integers (\cite{KS}, or see Lemma \ref{eigenvalues} below): let us partition the spectrum, $spec(\Lambda)$, into families $spec(\Lambda)=\mathcal{F}_1\cup\cdots\cup\mathcal{F}_k$ of algebraic 
conjugacy classes. For each $i\in\{1,\ldots,k\}$, let $d_i$ be the algebraic degree of the elements of $\mathcal{F}_i$, let $m_i:=max_{\lambda\in\mathcal{F}_i}m(\lambda)$, where $m(\lambda)$ is the multiplicity of  $\lambda$ as an eigenvalue of $\Lambda$, and let the {\bf degree} of $\Lambda$ be defined by $D(\Lambda):=\sum_{i=1}^km_id_i$. We will say that $\Phi$ is {\bf unimodular} if each element of $spec(\Lambda)$ is an algebraic unit, and {\bf hyperbolic} if no element of $spec(\Lambda)$ has an
algebraic conjugate on the unit circle.

If $\Phi$ is unimodular, there is a finitely generated subgroup $GR(\Phi)$ of $\R^n$, the group of {\bf generalized return vectors} (the definition is given below), that has rank $D=D(GR)\ge D(\Lambda)$ (Lemma \ref{eigenvalues}) and is invariant under $\Lambda$. Let $A$ denote the $D\times D$ matrix representing $\Lambda:GR(\Phi)\to GR(\Phi)$ with respect to some basis, and let $F_A:\T^D\to\T^D$ be the toral automorphism $x+\Z^D\mapsto Ax+\Z^D$ induced by $A$. Whenever $\Phi$ is hyperbolic, $F_A$ is also hyperbolic.

\begin{theorem} \label{main theorem} Suppose that $\Phi$ is unimodular and hyperbolic. There is then a finite-to-one continuous map $G:\Omega_{\Phi}\to\T^D$ so that 
$G\circ \Phi=F_A\circ G$. Furthermore, $G$ is $\mu$-a.e. $r$-to-one for some $r\in\N$ and $G$ is
homologically essential in that $G^*:H^1(\T^D;\Z)\to H^1(\OP;\Z)$ is injective.
\end{theorem}

Let $\Phi^*:H^1(\OP;\Z)\to H^1(\OP;\Z)$ denote the isomorphism induced by $\Phi$ on the integer \Cech cohomology of $\OP$. We will see that there is a $\Phi^*$-invariant subgroup $H^1_{\Lambda}$ of $H^1(\OP;\Z)$ restricted to which $\Phi^*$ is conjugate with the dual isomorphism $\Lambda^*:Hom(GR(\Phi);\Z)\to Hom(GR(\Phi);\Z)$. It often happens that $H^1_{\Lambda}$ is proper (see example \ref{example2} below), in which case it may be possible to lower the $r$ of Theorem \ref{main theorem} by increasing the size of the torus. Let $H^1_{hyp}$ denote the largest subgroup of $H^1(\OP;\Z)$ that contains $H^1_{\Lambda}$, is invariant under $\Phi^*$, and restricted to which 
$\Phi^*$ is unimodular and hyperbolic. Let $D'$ be the rank of $H^1_{hyp}$ and let $A'$ be the transpose of a matrix representing $\Phi^*:H^1_{hyp}\to H^1_{hyp}$ in some basis. 

\begin{theorem} \label{main theorem'} Suppose that $\Phi$ is unimodular and hyperbolic. There is then a finite-to-one, and $\mu$-a.e. $r'$-to-one for some $r'\le r$, map $G':\Omega_{\Phi}\to\T^{D'}$ so that 
$G'\circ \Phi=F_{A'}\circ G'$. Furthermore,  $(G')^*:H^1(\T^{D'};\Z)\to H^1(\OP;\Z)$ is injective with range $H^1_{hyp}$.
\end{theorem}

Theorems \ref{main theorem} and \ref{main theorem'} are proved in Section \ref{geometric realization}.

\begin{conjecture}\label{hypconjecture} If $\Phi$ is hyperbolic and $H^1_{hyp}=H^1(\OP;\Z)$ (i.e., $\Phi^*$ is unimodular and hyperbolic on all of $H^1(\OP;\Z$)), then $r'=1$ (that is, $G'$ is a.e. one-to-one).
\end{conjecture}

A family $\mathcal{F}_i$ of eigenvalues of $\Lambda$ is a {\bf Pisot family} provided all the elements of $\mathcal{F}_i$ have the same multiplicity as eigenvalues of $\Lambda$ and, if $\lambda$ is an algebraic conjugate of the elements of $\mathcal{F}_i$ with $|\lambda|\ge1$, then $\lambda\in\mathcal{F}_i$. We will say that $\Phi$ is a {\bf Pisot family substitution} if each $\mathcal{F}_i$ is a Pisot family.

For any group action on a space there is a maximal factor (unique up to conjugacy) on which the group acts equicontinuously. In the setting here, of $\R^n$-actions on tiling spaces $\Omega$, the maximal equicontinuous factor is a Kronecker action on a torus or solenoid. We will denote the maximal equicontinuous factor map by $g$. It is a consequence of the Halmos - von Neumann theory that the $\R^n$-action on $\Omega$ has pure discrete spectrum if, and only if, $g$ is $a.e.$
one-to-one (\cite{BK}).

\begin{Theorem}\label{gs=rp} Suppose that $\Phi$ is a unimodular Pisot family substitution with linear expansion $\Lambda$. If $D(\Lambda)=D(GR)=D$, then $G:\OP\to\T^D$ is surjective, semi-conjugates the $\R^n$-action on $\OP$ with a Kronecker action of $\R^n$ on $\T^D$, and the latter action is the maximal equicontinuous factor of the $\R^n$-action on $\OP$. That is, we may take $G=g:\OP\to\T^D$.
\end{Theorem}

\begin{conjecture}\label{homologicalPisot1} If $\Phi$ is a unimodular Pisot family substitution and $rank(H^1(\Omega_{\Phi};\Z))=D(\Lambda)$ then the $\R^n$-action on $\Omega_{\Phi}$ has pure discrete spectrum.
\end{conjecture}

There are counterexamples to Conjecture \ref{homologicalPisot1} if the assumption of unimodularity is dropped (see \cite{BBJS}, where a one-dimensional version of this conjecture is discussed). Our proofs of Theorems \ref{main theorem},  \ref{main theorem'}, and \ref{gs=rp} route through a certain non-compact cover of $\OP$ for which there doesn't appear to be an adequate analog in the non-unimodular case. J. Kwapisz has pointed out that results for non-unimodular hyperbolic substitutions can be obtained using sufficiently large, but compact, covers.

Theorem \ref{gs=rp}, and the Corollaries below, are proved in Section \ref{global shadowing}.

Given a tiling $T=\{\tau_i\}$, a {\bf puncture map} is a function $p:T\to\R^n$ so that $p(\tau_i)\in\tau_i$ and if $\tau_i=\tau_j+v$, then $p(\tau_i)=p(\tau_j)+v$. A set $\Gamma\subset\R^n$ is a {\bf Meyer set} if $\Gamma$ is relatively dense and $\Gamma-\Gamma$ is uniformly discrete. A tiling $T$ is said to have the Meyer property if for any (and hence all) puncture map(s) $p$, $p(T)$ is a Meyer set.

\begin{corollary}\label{g finite-to-one} Under the assumptions of Theorem \ref{gs=rp}:
\begin{itemize}
\item{$g$ is finite-to-one and a.e. $cr$-to-1;}
\item{the eigenvalues of the $\R^n$-action on $\OP$ are relatively dense in $\R^n$;}
\item{the tilings $T\in\OP$ have the Meyer property.}
\end{itemize}
\end{corollary}

\begin{cor}\label{hypcon implies Pcon} If Conjecture \ref{hypconjecture} is true, then so is Conjecture \ref{homologicalPisot1}.
\end{cor}

A Pisot family substitution is said to be of {\bf $(m,d)$-Pisot family type} if $\Lambda$ is diagonalizable over $\C$ and $spec(\Lambda)=\mathcal{F}_1$ consists of a single family with $m=m_1$ and $d=d_1$. Note in this case that $D(\Lambda)=md$. It is a result of \cite{BK} that, if $\Phi$ is of $(m,d)$-Pisot family type, then the maximal equicontinuous factor of the $\R^n$-action on $\OP$ is a Kronecker action on an $md$-dimensional torus (or solenoid, in case $\Phi$ is not unimodular). Moreover, the factor map $g:\OP\to\T^{md}$ is finite-to-one and a.e. $cr$-to-one for $cr=min\{\sharp g^{-1}(x):x\in\T^{md}\}$.

\begin{theorem} \label{Pisot} If $\Phi$ is unimodular of $(m,d)$-Pisot family type, then $G:\Omega_{\Phi}\to\T^D$ is the maximal 
equicontinuous factor of the $\R^n$-action. That is, $D(\Lambda)=D(GR)=md$ and $G=g$.
\end{theorem}

See Section \ref{Pisot family substitutions} for the proof.

\begin{remark} The traditional Pisot Substitution Conjecture is: If $\Phi$ is a 1-dimensional substitution with 
irreducible and unimodular incidence matrix and Pisot inflation factor, then the $\R$-action on $\Omega_{\Phi}$ 
has pure discrete spectrum. There are substitutions that satisfy these hypotheses and not those of Conjecture \ref{homologicalPisot1}, and vice-versa. In the final section of this paper the conditions in Conjectures \ref{hypconjecture} and \ref{homologicalPisot1} are relaxed a bit to produce  conjectures that do generalize the traditional Pisot Substitution Conjecture. The relaxation involves replacing $H^1(\OP;\Z)$ by the potentially smaller $H^1_{ess}(\OP;\Z)$.
\end{remark}

\begin{remark} Lee and Solomyak (\cite{LS}) show that if $\Phi$ is of $(m,d)$-Pisot family type, then the eigenvalues of the $\R^n$-action are relatively dense. The difficult part of their proof lies in demonstrating that the return vectors are contained in a $\Z$-module of rank $md=D(\Lambda)$. This is an assumption of Theorem \ref{gs=rp} above. But the condition $D(\Lambda)=D(GR)$ is easily verified in practice: It follows from Lemma \ref{eigenvalues} that if the multiplicity of $\lambda$ as an eigenvalue of $f_*:H_1(X;\Z)\to H_1(X;\Z)$ is the same as its multiplicity as an eigenvalue of $\Lambda$, for each $\lambda\in spec(\Lambda)$, then $D(\Lambda)=D(GR)$. (Here $f:X\to X$ is the map on the collared A-P complex, see below.)
\end{remark}

The cohomologies (and {\it essential} cohomologies - see Section \ref{finalsection}) in the following examples are easily computed by the methods of \cite{bd1}.

\begin{example}\label{example1} Let $\phi$ be the substitution on letters: $a\mapsto ab',\, b\mapsto a,\, a'\mapsto a'b,\, b'\mapsto a'$. The corresponding one-dimensional substitution $\Phi$ is of $(1,2)$-Pisot family type. $H^1_{hyp}=H^1_{\Lambda}\simeq \Z^2$, so $G'=G=g$. These maps are a.e. 2-to-1 and 
$H^1_{hyp}$ is proper in $H^1_{ess}(\OP;\Z)\simeq\Z^4$. (Here $H^1(\OP;\Z)\simeq\Z^7$.)
\end{example}

\begin{example}\label{example2} Let $\phi$ be the substitution on letters: $a\mapsto aba',\, b\mapsto ab,\,a'\mapsto a'b'a,\, b'\mapsto a'b'$.The corresponding one-dimensional substitution $\Phi$ is of $(1,2)$-Pisot family type so $G=g:\OP\to\T^2$ and these maps are a.e. 2-to-1. $H^1(\OP;\Z)\simeq\Z^5$, but $H^1_{hyp}=H^1_{ess}(\OP;\Z)\simeq\Z^4$ and $G':\OP\to\T^4$ is a.e. 1-to-1.
\end{example}

\section{Abelian covers of $\OP$ and global shadowing}
To simplify notation, let us fix a (primitive, FLC, non-periodic) $n$-dimensional substitution $\Phi$ and let $\Omega:=\OP$.
Let $X$ be the Anderson-Putnam complex (A-P complex) for $\Phi$ (see \cite{AP}):  $X$ is a cell complex with one n-cell $\rho_i\times\{i\}$ for each prototile $\rho_i$, and these n-cells are glued along faces according to the following scheme.
Suppose that $\rho_i$ and $\rho_j$ are prototiles and $u,v\in\R^n$, $T\in\Omega$, are such that
$\rho_i+u,\rho_j+v\in T$. Set $\rho_i\times\{i\}\ni (x,i)\sim (y,j)\in\rho_j\times\{j\}$ if $x+u=y+v$ and 
extend $\sim$ to an equivalence relation on $\cup_i\rho_i\times\{i\}$. The A-P complex is the quotient $X:=\cup_i\rho_i\times\{i\}/\sim$. There is a natural map $p:\Omega\to X$ given by $p(T)=
[(x,i)]_{\sim}$ provided $0=u+x\in u+\rho_i\in T$ and an induced map $f:X\to X$ with $p\circ \Phi=
f\circ p$. The map $\hat p:\Omega\to\il f$ given by $\hat p(T):=(p(T),p(\Phi^{-1}(T)),\ldots)$ and the shift homeomorphism $\hat f:\il f\to\il f$ satisfy:

\begin{itemize}
\item{$\hat p\circ\Phi=\hat{f}\circ p$; and}
\item{$\hat p$ is an a.e. (with respect to $\mu$ on $\Omega$) one-to-one surjection.}

\end{itemize}
Moreover, when 
$\Phi$ has the property that it ``forces the border", which can be arranged by replacing the tiles of tilings in $\Omega$ by their collared versions, the map $\hat p$ is a homeomorphism (\cite{AP}).

Let $\pi:\tilde{X}\to X$ be the universal cover of the A-P complex. The group of deck transformations of $\tX$ can be identified with the fundamental group of $X$. We form the {\bf abelian cover} $\pi_{ab}:\tX_{ab}\to X$ by quotienting out the action of the commutator subgroup, $C$, of the fundamental group of $X$: $\tX_{ab}:=\tX/\sim_{ab}$ with $\tx\sim_{ab}\ty$ if and only if there is $\gamma\in C$ with $\gamma(\tx)=\ty$. Since the map $f_{\sharp}$ induced by $f$ on the fundamental group of $X$ takes $C$ into $C$, $f$ lifts to $\tf_{ab}:\tX_{ab}\to\tX_{ab}$. The group of deck transformations of $\tX_{ab}$ can be identified with the first homology $H_1(X;\Z)$. If $h\in H_1(X;\Z)$ and $\tx\in\tX_{ab}$ we write $\tx+h$, or sometimes $h(\tx)$, for the image of $\tx$ under the deck transformation corresponding to $h$, and we have $\tf_{ab}(\tx+h)=\tf(\tx)+f_*(h)$ with $f_*$ the homomorphism induced on $H_1(X;\Z)$ by $f$.

Suppose that $K$ is any subgroup of $H_1(X;\Z)$. There is then a corresponding cover $\pi_K:\tX_K\to X$ where $\tX_K$ is the quotient of $\tX_{ab}$ by the action
of $K$. The group of deck transformations of $\tX_K$ is $H_1(X;\Z)/K$ and if $K$ is invariant under $f_*$, $f$ lifts to $\tf_K:\tX_K\to\tX_K$, and $\tf_K(\tx+[h])=\tf_K(\tx)+[f_*(h)]$, where $[h]$ denotes the coset $h+K$.

\begin{lemma}\label{covering map} Suppose that the subgroup $K$ of $H_1(X;\Z)$ is invariant under $f_*$ and that $\bar{f}_*$ defined by $\bar{f}_*([h]):= [f_*(h)]$ is an isomorphism of $H_1(X;\Z)/K$. Then the natural map $\hat{\pi}_K:\il \tf_K\to\il f$, given by $\hat{\pi}_K((\tx)_i):=(\pi_K(\tx_i))$, is a covering map with group of deck transformations equal to $H_1(X;\Z)/K$.
\end{lemma}
\begin{proof} The issue is surjectivity. Given $x,y\in X$ with $f(x)=y$, and $\ty\in\tX_K$ with $\pi_K(\ty)=y$, pick $\tx'\in\tX_K$ with $\pi_K(\tx')=x$. Then $\pi_K(\tf_K(\tx'))=y$ so there is $h\in H_1(X;\Z)$ with  $\tf_K(\tx')+[h]=\ty$. Let $\tx:=\tx'+\bar{f}_*^{-1}([h])$. Then 
$\tf_K(\tx)=\ty$ and surjectivity of $\hat{\pi}_K$ follows.  
\end{proof}

Let us give another description of the covering $\hat{\pi}_K:\il \tf_K\to\il f$. Let $\Omega=\OP$ and let 
$\tilde{\Omega}_K:=\{(T,\tx):p(T)=\pi_K(\tx)\}\subset \Omega\times\tX_K$, with the product topology.
Let $\pi_1:\tilde{\Omega}_K\to\Omega$ and $\pi_2:\tilde{\Omega}_K\to\tX_K$ be projections onto first and second factors.

\begin{lemma}\label{tO} Suppose, as in Lemma \ref{covering map}, that $K$ is invariant under $f_*$ and that $\bar{f}_*$ is an isomorphism. Suppose also that $X$ is the collared A-P complex of $\Phi$. Then $\pi_1:\tilde{\Omega}_K\to\Omega$ is isomorphic with $\hat{\pi}_K:\il\tf_K\to\il f$.
\end{lemma}
\begin{proof} Since $X$ is collared, the map $\hat{p}:\Omega\to \il f$, $\hat{p}(T):=(p(T),p(\Phi^{-1}(T),\ldots)$, is a homeomorphism. The continuous map $\tilde{\hat{p}}$ defined by $(T,\tx)\mapsto (\tx_i)$, where $\tx_i$ satisfies: $\tx_0=\tx$ and $\pi_K(\tx_i)=p(\Phi^{-i}(T))$ has continuous inverse $(\tx_i)\mapsto (\hat{p}^{-1}((\pi_K(\tx_i))),\tx_0)$. Moreover, $\hat{\pi}_K\circ \tilde{\hat{p}}=\hat{p}\circ\pi_1$.
\end{proof}

Under the hypotheses of Lemma \ref{tO}:
\begin{itemize}
\item We can lift the homeomorphism $\Phi$ on $\Omega$ to a homeomorphism $\tilde{\Phi}_K:\tilde{\Omega}_K\to\tilde{\Omega}_K$ defined by $\tilde{\Phi}_K((T,\tx)):=(\Phi(T),\tf_K(\tx))$. This homeomorphism is conjugated with $\hat{\tf}_K$, the shift homeomorphism on  $ \il \tf_K$, by $\tilde{\hat{p}}$.
\item We may also lift the $\R^n$-action from $\Omega$ to $\tilde{\Omega}_K$ as follows. Given $T\in\Omega$ let 
$p^T:\R^n\to X$ be defined by $p^T(v):=p(T-v)$. Given $\tx\in\tX_K$ with $\pi_K(\tx)=p(T)$, let $\tilde{p}^T_{\tx}:\R^n\to\tX_K$ be the unique lift of $p^T$ satisfying $\tilde{p}^T_{\tx}(0)=\tx$. For $\tT=(T,\tx)\in\tilde{\Omega}_K$ and $v\in\R^n$ define $\tT-v:=(T-v,\tilde{p}^T_{\tx}(v))$. Note that $\tilde{\Phi}_K(\tT-v)=\tilde{\Phi}_K(\tT)-\Lambda v$.
\item Finally we equip $\tilde{\Omega}_K$ with a metric $\bar d$ as follows.
First we select a metric $\td$ on  $\tX_K$ with two important properties. Suppose that
$\{[h_i]\}_{i=1}^N$ is a basis for $H_1(X;\Z)/K$. For $[h]=\sum_{i=1}^Nb_i[h_i]$, let $|[h]|:=\sum_{i=1}^N|b_i|$. We take $\td$ so that:
\begin{enumerate}
\item  $\td(\tx+[h],\ty+[h])=\td(\tx,\ty)$ for all $\tx,\ty\in\tX_K, [h]\in H_1(X;\Z)/K$; and
\item  $\td(\tx,\ty+[h])\to\infty$ as $|[h]|\to\infty$ for all $\tx,\ty\in\tX_K$.
\end{enumerate}
For the metric on $\tilde{\Omega}_K$ we set $\bar{d}((T,\tx),(S,\ty)):=d(T,S)+\td(\tx,\ty)$. \end{itemize}
For $T,S\in\Omega$, we say that
$T$ {\bf globally shadows} $S$ (with respect to K), and write $T\sim_{gsK}S$, if there are $\tT=(T,\tx)$ and $\tS=(S,\ty)$ in $\tO_K$ so that
$\{\bar{d}(\tilde{\Phi}_K^k(\tT),\tilde{\Phi}_K^k(\tS)\}_{k\in\Z}$, is bounded.

\section{Generalized return vectors}

A vector $v\in\R^n$ is called a {\bf return vector} for $\Phi$ if there is $T\in\OP$ so that $p(T-v) =p(T)$.
Let us call a vector $v\in\R^n$ a {\bf generalized return vector} for $\Phi$ if there are $v_i\in\R^n$ and $T_i\in\Omega$, $i=1,\ldots,k$, with $v=v_1+\cdots+v_k$, so that $p(T_i-v_i)=p(T_{i+1})$ for $i=1,\ldots,k-1$,
and $p(T_k-v_k)=p(T_1)$. The collection of generalized return vectors is a subgroup of $\R^n$ which we denote by  $GR(\Phi)$. If $\Phi$ is unimodular, one can show that $GR(\Phi)$ equals the subgroup of $\R^n$ generated by the return vectors.

Any path $\alpha:[0,1]\to X$ in the Anderson-Putnam complex $X$ for $\Phi$ can be ``lifted" to a curve $\tilde{\alpha}:[0,1]\to\R^n$ (think of unfolding the tiles that $\alpha$ runs through). Let
$l(\alpha):=\tilde{\alpha}(1)-\tilde{\alpha}(0)$.

\begin{lemma}\label{return vector} (\cite{BSW}) If $\alpha$ is a path in the A-P complex $X$ for $\Phi$, the vector $l(\alpha)$ is a well-defined (independent of lift) element of $GR(\Phi)$ and depends only on the homotopy class (rel. endpoints) represented by $\alpha$. Furthermore, if $\alpha$ is a loop, $l(\alpha)$ depends only on the homology class of $\alpha$. The resulting function $l:H_1(X;\Z)\to GR(\Phi)$ is a surjective group homomorphism and $l\circ f_*=\Lambda l$.
\end{lemma} 

Let $K_{\Lambda}$ be the kernel of $l:H_1(X;\Z)\to GR(\Phi)$. By Lemma \ref{return vector}, $K_{\Lambda}$ is invariant under $f_*$; let $\bar{f}_*:H_1(X;\Z)/K_{\Lambda}\to H_1(X;\Z)/K_{\Lambda}$ be the homomorphism induced by $f_*$. The group
$H_1(X;\Z)/K_{\Lambda}\simeq GR(\Phi)$ is a finitely generated ($H_1(X;\Z)$ is finitely generated) free
abelian group ($GR(\Phi)$ is a subgroup of $\R^n$) so, with the choice of some basis, $\bar{f}_*$ is represented by an integral matrix $A$. The following lemma is an adaptation of a result in \cite{KS}.

\begin{lemma}\label{eigenvalues} Let $spec(\Lambda)=\cup_{i=1}^k\mathcal{F}_i$ be the partition of $spec(\Lambda)$ into families of algebraically conjugate eigenvalues and let $m_i$ be the maximum multiplicity of the elements of $\mathcal{F}_i$. Then $\lambda$ is an eigenvalue of $A$ of multiplicity $m$ if and only if $\lambda$ has a conjugate in $\mathcal{F}_i$ for some $i$. Moreover, $m\ge m_i$.
\end{lemma}

\begin{proof} Let $\{[h_1],\ldots,[h_k]\}$ be a basis for the $\Z$-module $H_1(X;\Z)/K_{\Lambda}$ with respect to which $\bar{f}_*$ is represented by $A$ and let $v_i:=l(h_i)\in GR(\Phi)$, $i=1,\ldots,k$. Let $L:\R^k\to\R^n$ be the linear map that takes the standard basis vector $e_i$ to $v_i$ for each $i$. Then $LA=\Lambda L$ and,
since the return vectors span $\R^n$ (a consequence of minimality of the $\R^n$-action on $\Omega$),
$L$ is surjective. Suppose that $\lambda$ has a conjugate $\mu$  in $\mathcal{F}_i$ for some $i$ 
and that $\mu$ is a real eigenvalue of  $\Lambda$ with multiplicity $m_i$.
There is then a $\Lambda^t$-invariant $m_i$-dimensional subspace $V\subset\R^n$, restricted to which $(\Lambda^t-\mu I)^{m_i}$ is zero. Then $(A^t-\mu I)^{m_i}$ is zero on the $m_i$-dimensional subspace $L^tV$ of $\R^k$. Since $L^t$ is injective, this means that $\mu$ is an eigenvalue of $A$ of multiplicity at least $m_i$, and thus that $\lambda$ is an eigenvalue with multiplicity at least $m_i$ of $A$.  We leave it to the reader to reach the same conclusion when $\lambda$ is complex.  

Now suppose that $\lambda$ is an eigenvalue of $A$. Let $p(t)=r(t)q(t)$ be the characteristic polynomial of $A$ factored with $r(t),q(t)\in\Z[t]$ so that all roots of $q(t)$, and no roots of $r(t)$, are algebraic conjugates of $\lambda$. Then $W:=ker(q(A))$ is a non-trivial subspace of $\R^k$, invariant under $A$, restricted to which all eigenvalues of $A$ are conjugates of $\lambda$. Moreover, since $q$ and $A$ are integer, $W$ has a basis $\{w_1,\ldots,w_d\}$ with $w_i\in \Z^k$ for each $i$. Let $w_i=\sum_{j=1}^dw_{ij}e_j$ with $w_{ij}\in\Z$. Then the elements $[h'_i]\in H_1(X;\Z)/K_{\Lambda}$, $h'_i:=\sum_{j=1}^dw_{ij}h_i$, span a $\bar{f}_*$-invariant submodule of $H_1(X;\Z)/K_{\Lambda}$
of rank $d>0$. Since $h'_i\notin K_{\Lambda}$, $l(h'_i)\ne0$ and it follows that $Lw_i\ne0$. Thus $L(W)$ is a nontrivial subspace of $\R^n$, invariant under $\Lambda$. Applying the argument of the previous paragraph to $A|_W$, $L|_W$ and $\Lambda|_{L(W)}$, we conclude that $\Lambda$ has an eigenvalue that is also an eigenvalue of $A|_W$, and is hence a conjugate of $\lambda$.
\end{proof}

From Lemma \ref{eigenvalues} we have:

\begin{corollary}  $D(GR)\ge D(\Lambda)$.
\end{corollary}

\begin{question} It is a consequence of Theorem 3.1 of \cite{LS} that, if $\Lambda$ is diagonalizable over $\C$ and all of its eigenvalues are algebraic conjugates of the same multiplicity, then $D(GR)=D(\Lambda)$. Is it ever the case that $D(GR)>D(\Lambda)$?
\end{question}

We see from Lemma \ref{eigenvalues} that if $\Phi$ is unimodular and hyperbolic, then $\bar{f}_*$ is unimodular and hyperbolic on $H_1(X;\Z)/K_{\Lambda}$. It is sometimes possible to increase the size of the quotient module while retaining the unimodularity and hyperbolicity of the corresponding quotient isomorphism (example \ref{example2}). Let $T$ denote the torsion subgroup of $H_1(K;\Z)$, let $f_*'$ denote the induced homomorphism $f_*':H_1(X;\Z)/T\to H_1(X;\Z)/T$ of the finitely generated free abelian group $H_1(X;\Z)/T$, and let $A$ be the matrix representing $f_*'$ in some basis. We may then factor the characteristic polynomial $p(t)$ of $A$, over $\Z$, as $p(t)=q(t)r(t)$ so that all roots of $q(t)$ are algebraic units and none have modulus 1; and all roots of $r(t)$ are not units, or have a conjugate of modulus 1. Let $F_0:=Ker(r(f_*'))$ and let $K_{hyp}:=\pi^{-1}(F_0)\subset H_1(X;\Z)$, where $\pi:H_1(X;\Z)\to H_1(X;\Z)/T$ is the quotient homomorphism. Then
$K_{hyp}\subset K_{\Lambda}$ is invariant under $f_*$, is independent of the choice of basis, and is minimal with respect to the property that the induced isomorphism $\bar{f}_*:H_1(X;\Z)/K_{hyp}\to
H_1(X;\Z)/K_{hyp}$ is unimodular and hyperbolic.

The group $H^1(X;\Z)$ is (naturally) isomorphic with $Hom(H_1(X;\Z),\Z)$ so the \Cech cohomology $H^1(\il f;\Z)=\dl f^*:H^1(X;\Z)\to H^1(X;\Z)$ is isomorphic with the direct limit of the dual of $f_*$. Let $H^1_{hyp}:=\{c\in H^1(X;\Z): c(h)=0 \text{ for all }h\in K_{hyp}\}$. Then $f^*|_{H^1_{hyp}}:H^1_{hyp}\to H^1_{hyp}$ is an isomorphism so that $H^1_{hyp}\simeq \dl f^*|_{H^1_{hyp}}$.
By means of $\hat{p}^*:\dl f^*\to H^1(\OP;\Z)$, $H^1_{hyp}$ can be viewed as a subgroup of $H^1(\OP;\Z)$; if $X$ is collared (so that $\hat{p}^*$ is an isomorphism), $H^1_{hyp}\subset H^1(\OP;\Z)$ is the largest subgroup of $H^1(\OP;\Z)$ that is invariant under $\Phi^*$, contains $H^1_{\Lambda}$, and on which $\Phi^*$ is unimodular and hyperbolic.

\section{From substitutions to toral automorphisms}

The first cohomology group $H^1(X;\Z)$ is naturally isomorphic with the {\bf Bruschlinski group}
consisting of homotopy classes of maps from $X$ to the additive circle group $\T:=\R/\Z$ (that is, $\T$ is a $K(\Z,1)$). An explicit isomorphism is given by $\Theta([\gamma])= \gamma^*({\bf1})$, where $\bf{1}$ is the fundamental class of $H^1(\T;\Z)$ and $\gamma^*(\bf{1})$ is its pullback to $H^1(X;\Z)$. Let us fix a subgroup $K$ of $H_1(X;\Z)$, invariant under $f_*$, and define $H^1_K:=\{c\in Hom(H_1(X;\Z),\Z):c(h)=0$ for all $h\in K\}\subset Hom(H_1(X;\Z),\Z)\simeq H^1(X;\Z)$.  Let $\{c_1,\ldots,c_N\}$ be a basis for $H^1_K$ and, for each $i=1,\ldots,N$, choose a map $\gamma_i$ with $[\gamma_i]:=\Theta^{-1}(c_i)$.
Then, for each $i$, $\Theta([\gamma_i])=(\gamma_i^*)^{-1}({\bf1})= c_i$ annihilates $K$.

For the following proposition, let $A=(a_{i,j})$ denote the transpose of the matrix representing the homomorphism induced on 
$H^1_K$ by $f$ with respect to the basis $\{c_1,\ldots,c_N\}$. Thus $\gamma_j\circ f=
\sum_{i=1}^Na_{j,i}\gamma_i$, up to homotopy. We also assume that $X$ is collared so that 
$\il f\simeq \OP$ (although the proposition would still be true in general, replacing $\OP$ by $\il f$  
and making the appropriate adjustments). Let $F_A:\T^N\to\T^N$ be the linear torus map $F_A(x+\Z^N):=Ax+\Z^N$, let $\Gamma:X\to\T^N$ be given by $\Gamma(x):=(\gamma_1(x),\ldots,\gamma_N(x))^t$, and let $G_0:=\Gamma\circ p:\OP\to\T^N$.

\begin{Proposition}\label{main prop} Suppose that $\Phi$ is unimodular and hyperbolic, and that $K$ is a subgroup of $H_1(X;\Z)$ that is invariant under $f_*$ and lies between $K_{hyp}$ and $K_{\Lambda}$. Then the map $F_A$ is a hyperbolic toral automorphism and there exists a continuous map $G:\OP\to \T^N$ with the properties:

(i) $G^*=G_0^*:H^1(\T^N;\Z)\to H^1(\OP;\Z)$;

(ii) $G\circ \Phi=F_A\circ G$;

(iii) $G(T)=G(S)$ if and only if $T\sim_{gsK}S$.
\end{Proposition}

For the following lemma, let $\tilde{\Omega}\to\Omega$ be a covering map with group of deck transformations $H$.

\begin{lemma}\label{same in cohomology} Suppose that $G_i:\Omega\to\T^N,i=1,2$, are continuous maps with lifts $\tilde{G}_i
:\tilde{\Omega}\to\R^N$. If $|\tilde{G}_1-\tilde{G}_2|$ is bounded, then $G_1^*=G_2^*:H^1(\T^N;\Z)\to
H^1(\Omega;\Z)$.
\end{lemma} 
\begin{proof} There are homomorphisms $\alpha_i:H\to\Z^N$ so that $\tilde{G}_i(\tilde{x}+h))=\tilde{G}_i(\tilde{x})+\alpha_i(h)$ for all $h\in H$ and $i=1,2$. Clearly, $|\tilde{G}_1-\tilde{G}_2|$ bounded implies that $\alpha_1=\alpha_2=:\alpha$. Suppose that $\gamma:\T^N\to\T$ is continuous. Let $\tilde{\gamma}:\R^N\to\R$ be a lift. Define $\tilde{H}(\tilde{x},t):=t(\tilde{\gamma}\circ\tilde{G}_1(\tilde{x}))+(1-t)(\tilde{\gamma}\circ\tilde{G}_2(\tilde{x}))$. Then $\tilde{H}(\tilde{x}+h,t)=\cdots=\tilde{H}(\tilde{x},t)+\alpha(h)$, so
that $\tilde{H}$ descends to a homotopy from $\gamma\circ G_2$ to $\gamma\circ G_1$. Thus
$G_1^*(\gamma^*(\bf{1}))=G_2^*(\gamma^*(\bf{1}))$. That is, $G_1^*\circ\Theta=G_2^*\circ\Theta$,
so $G_1^*=G_2^*$, as $\Theta:[\T^N,\T]\to H^1(\T^N;\Z)$ is an isomorphism.
\end{proof}

\begin{proof}(of Proposition \ref{main prop}) 
Suppose that $\alpha$ is a loop in $\tX_K$. Then $\pi_K\circ\alpha$ is a loop that represents a homology class $[\pi_K\circ\alpha]$ in $K$ and $c_i([\pi_K\circ\alpha])=0$. That is, $(\gamma_i^*)^{-1}({\bf1})(([\pi_K\circ\alpha]))={\bf1}((\gamma_i)_*([\pi_K\circ\alpha])={\bf 1}([\gamma_i\circ\pi_K\circ\alpha])=0$. This means that the loop $\gamma_i\circ\pi_K\circ\alpha$ in $\T$ is null homotopic and it follows that the map $\gamma_i:X\to\T$ lifts
to $\tilde{\gamma}_i:\tX_K\to\R$. Recall that $\Gamma:X\to\T^N$ is given by $\Gamma(x):=(\gamma_1(x),\ldots,\gamma_N(x))^t$; then $\tilde{\Gamma}:=(\tilde{\gamma}_1,\ldots,\tilde{\gamma}_N)^t:\tX_K\to\R^N$ is a lift of $\Gamma$ and $\tilde{\Gamma}_0:=
\tilde{\Gamma}\circ\pi_0:\il \tf_K\to\R^N$ is a lift of $G_0=\Gamma\circ p$ (here $\pi_0$ is projection onto the zeroth coordinate and we have identified $\tO_K$ with $\il \tf_K$ via Lemma \ref{tO}).

 \noindent We know from Lemma \ref{eigenvalues}  that the integer matrix $A$ is unimodular and  hyperbolic and thus the linear torus map $F_A$ is a hyperbolic toral automorphism.
The homomorphism $\Gamma_*$ takes $K$ to 0 and the induced homomorphism $\bar{\Gamma}_*:H_1(X;\Z)/K\to H_1(\T^N;\Z)$ is an isomorphism. Moreover, the choice of $A$ guarantees that $\bar{\Gamma}_*\circ \bar{f}_*=(F_A)_*\circ\bar{\Gamma}_*$. 
Since $\tilde{\Gamma}$ is a lift of $\Gamma$, we have $\tilde{\Gamma}(\tx+[h])=\tilde{\Gamma}(\tx)+\bar{\Gamma}_*([h])$ for $\tx\in\tX_K$ and $[h]\in H_1(X;\Z)/K$. Thus $A\bar{\Gamma}_*=\bar{\Gamma}_*\circ\bar{ f}_*$ (we have identified $\Z^N\subset\R^N$ with $H_1(\T^N;\Z)$). We see then that 
$A\tilde{\Gamma}_0((\tx_i)+[h])=A(\tilde{\Gamma}_0((\tx_i))+\bar{\Gamma}_*([h]))=A\tilde{\Gamma}_0((\tx_i))+\bar{\Gamma}_*\circ \bar{f}_*([h])$, while $\tilde{\Gamma}_0\circ\hat{\tf}_K((\tx_i)+[h])=\tilde{\Gamma}_0(\hat{\tf}_K((\tx_i)+\bar{f}_*([h]))=\tilde{\Gamma}_0\circ\hat{\tf}_K((\tx_i))+\bar{\Gamma}_*\circ \bar{f}_*([h])$. It follows that $|A\tilde{\Gamma}_0-\tilde{\Gamma}_0\circ \hat{\tf}_K|$ is uniformly bounded on $\il \tf_K$ (by its bound on a single fundamental domain).

\noindent Let $E^s$ and $E^u$ denote the stable and unstable linear subspaces of $\R^N$ under application of $A$. Then $\R^N=E^s\oplus E^u$, and $E^s$ and $E^u$ are invariant under $A$. For each $z\in\R^N$,
let $z^s\in E^s$ and $z^u\in E^u$ be so that $z=z^s+z^u$. There are $C$ and $\eta$, $0<\eta<1$, with $||A^kz^s||\le C\eta^k||z^s||$ and  $||A^{-k}z^u||\le C\eta^{-k}||z^u||$ for all $z\in\R^N$ and $k\in\N$.
Now, given $(\tx_i)\in\il\tf_K$ and $k\in\Z$, let $y_k:=\tilde{\Gamma}_0(\tf_K^k((\tx_i)))$. By the above, $b_k:=y_{k+1}-Ay_k$ is bounded. For $k\in\N$ we have $A^{-k}y_k=y_0+\sum_{i=1}^kA^{-i}b_{i-1}$ and 
$A^ky_{-k}=y_0-\sum_{i=0}^{k-1}A^ib_i$. Define $z=z((\tx_i))$ by $z=z^u+z^s$ where
$$z^u:=\lim_{k\to\infty}(A^{-k}y_k)^u=y_0^u+\sum_{i=1}^{\infty}A^{-k}b_{k-1}^u$$ and 
$$z^s:=\lim_{k\to\infty}(A^ky_{-k})^s=y_0^s-\sum_{i=0}^{\infty}A^kb_{-k}^s.$$
It is clear that $z$ depends continuously on $(\tx_i)$, that $z((\tx_i)+[h])=z((\tx_i))+\bar{\Gamma}_*([h])$, and that $|\tilde{\Gamma}_0((\tx_i))-z((\tx_i))|$ is bounded. Thus the map $\tilde{\Gamma}':\il\tf_K\to\R^N$ given by
$\tilde{\Gamma}'((\tx_i)):=z((\tx_i))$ is the lift of a continuous map $\Gamma':\il f\to\T^N$.
Let $\tilde{\hat{p}}:\tO_K\to\il\tf_K$ be the isomorphism of Lemma \ref{tO}, let $\tG:=\tilde{\Gamma}'\circ\tilde{\hat{p}}:\tO_K\to\R^N$ and let $G:=\Gamma'\circ\hat{p}:\Omega\to\T^N$.
By Lemma \ref{same in cohomology}, $G^*=G_0^*$. Furthermore, $(\tilde{\Gamma}'(\hat{\tf}_K((\tx)_i)))^u=(\tilde{\Gamma}'((\tx_{i+1})))^u=\lim_{k\to\infty}(A^{-k}y_{k+1})^u=A\lim_{k\to\infty}(A^{-(k+1)}y_{k+1})^u=A(\tilde{\Gamma}'((\tx_i)))^u$. Similarly, $(\tilde{\Gamma}'(\hat{\tf}_K((\tx_i))))^s=A(\tilde{\Gamma}'((\tx)_i)))^s$, so that $\tilde{\Gamma}'\circ\hat{\tf}_K=A\tilde{\Gamma}'$, whence $\Gamma'\circ \hat{f}=F_A\circ \Gamma'$. Since $\hat{f}\circ\hat{p}=\hat{p}\circ\Phi$, we have 
$G\circ\Phi=F_A\circ G$.

It remains to prove (iii). First suppose that $T\sim_{gsK}S$. Let $\tT$ and $\tS$, lying over $T$ and $S$, have the property that $\bar{d}(\tilde{\Phi}_K^k(\tT),\tilde{\Phi}_K^k(\tS))$ is bounded for $k\in\Z$. As $\tilde{G}$ is a lift of a continuous function on a compact space, and the metric $\bar{d}$ is equivariant, $\tilde{G}$ is uniformly continuous and hence $|\tilde{G}(\tilde{\Phi}_K^k(\tT))-\tilde{G}(\tilde{\Phi}_K^k(\tS))|=|A^k(\tilde{G}(\tT)-\tilde{G}(\tS))|$ is also bounded. Since $A$ is hyperbolic, this can only happen if $\tilde{G}(\tT)=\tilde{G}(\tS)$. Thus $G(T)=G(S)$.

\noindent Conversely, if $G(T)=G(S)$, let $\tT$ and $\tS$ lie over $T$ and $S$. There is then $h\in\Z^N=H_1(\T^N;\Z)$ so that $\tilde{G}(\tT)+h=\tilde{G}(\tS)$ and there is $h'\in H_1(X;\Z)/K$ so that $\bar{\Gamma}_*(h')=h$. Let  $\tT':=\tT+h'$. Then $\tT'$ also lies over $T$ and $\tilde{G}(\tT')=\tilde{G}(\tS)$. From $\tilde{G}\circ\tilde{\Phi}_K=A\tilde{G}$ it follows that 
$\tilde{G}(\tilde{\Phi}_K^k(\tT'))=\tilde{G}(\tilde{\Phi}_K^k(\tS))$ for all $k\in\Z$. Were $\bar{d}(\tilde{\Phi}_K^k(\tT'),\tilde{\Phi}_K^k(\tS))$ not bounded, there would be $k_j\in\Z$ and $h_j\in H_1(X;\Z)/K$ with $\bar{d}(\tilde{\Phi}_K^{k_j}(\tT')+h_j,\tilde{\Phi}_K^{k_j}(\tS))$ bounded and $|h_j|\to\infty$ (see the second of the assumptions on the nature of $\td$), hence $|\bar{\Gamma}_*(h_j)|\to\infty$. But (by uniform continuity of $\tilde{G}$ and equivariance of $\bar{d}$), $|\tilde{G}(\tilde{\Phi}_K^{k_j}(\tT')+h_j)-\tilde{G}(\tilde{\Phi}_K^{k_j}(\tS)|=|\bar{\Gamma}_*(h_j)|$ is bounded. Thus it must be the case that $T\sim_{gsK}S$.
\end{proof}

\section{Geometric realization}\label{geometric realization}

\begin{proof}(of Theorems \ref{main theorem} and \ref{main theorem'}) 
Let $K=K_{\Lambda}$ for Theorem \ref{main theorem} or $K=K_{hyp}$ for Theorem \ref{main theorem'} and let $A$, $G$, $N$ be as in Proposition \ref{main prop}, with $N=D$ or $N=D'$, depending on whether $K=K_{\Lambda}$ or $K=K_{hyp}$ . Both $(\T^N,F_A)$ and $(\Omega,\Phi)$ are Smale spaces - the actions are hyperbolic with local product structure (see \cite{AP}). A lemma of Putnam (\cite{P}) asserts that a factor map (such as $G$) between Smale spaces
that is injective on unstable manifolds is globally finite-to-one. 

Thus we are reduced to proving that $G$ is injective on unstable manifolds.
Given $T\in\Omega$, the unstable manifold of $T$ under $\Phi$ is the set $W^u(T):=\{T'\in\Omega:d(\Phi^k(T),\Phi^k(T'))\to0$ as $k\to-\infty\}$. It is easy to see that $W^u(T)=\{T-v:v\in\R^n\}$. 

Suppose that $G(T-v)=G(T)$ for some $T\in\Omega$ and $0\ne v\in\R^n$. According to Proposition
\ref{main prop}, $T\sim_{gsK}T-v$. Let us first show that the lifts $\tT=(T,\tx)$ and $\tT-v=(T-v,\tilde{p}^T_{\tx}(v))$ of $T$ and $T-v$ to $\tO_K$ are such that $\bar{d}(\tilde{\Phi}_K^k(\tT),\tilde{\Phi}_K^k(\tT-v))\to\infty$ 
as $k\to\infty$. Note that  $\pi_2(\tilde{\Phi}_K^k(\tT))=\tilde{p}^{\Phi^k(T)}_{\tf_K^k(\tx)}(0)$ and $\pi_2(\tilde{\Phi}_K^k(\tT-v))=\tilde{p}^{\Phi^k(T)}_{\tf_K^k(\tx)}(\Lambda^k v)$.

We may choose return vectors $v_k$ so that $|v_k-\Lambda^kv|$ is bounded and $B_r[\Phi^k(T)]=B_r[\Phi^k(T)-v_k]$ with $r$ twice the maximum diameter of all tiles (recall that  $B_r[T]$ denotes the collection of all tiles in $T$ that meet the closed ball centered at $0$ with radius $r$). Then $p(\Phi^k(T))=p(\Phi^k(T)+v_k)$ in the collared Anderson-Putnam complex $X$. There is $[h_k]\in H_1(X;\Z)/K$ with $l(h_k)=v_k$. Note that there are infinitely many distinct such $v_k$, and hence infinitely many distinct $[h_k]$ for $k\in\N$. It follows that
$$\td(p^{\Phi^k(T)}_{\tf_K^k(\tx)}(0),p^{\Phi^k(T)}_{\tf_K^k(\tx)}(0)+[h_k])=\td(p^{\Phi^k(T)}_{\tf_K^k(\tx)}(0), p^{\Phi^k(T)}_{\tf_K^k(\tx)}(v_k))$$ is unbounded for $k\in\N$. Hence $$\td(p^{\Phi^k(T)}_{\tf_K^k(\tx)}(0),p^{\Phi^k(T)}_{\tf_K^k(\tx)}(\Lambda^kv))\ge \td(p^{\Phi^k(T)}_{\tf_K^k(\tx)}(0),p^{\Phi^k(T)}_{\tf_K^k(\tx)}(v_k))-\td(p^{\Phi^k(T)}_{\tf_K^k(\tx)}(v_k),p^{\Phi^k(T)}_{\tf_K^k(\tx)}(\Lambda^kv))$$$$\ge \td(p^{\Phi^k(T)}_{\tf_K^k(\tx)}(0),p^{\Phi^k(T)}_{\tf_K^k(\tx)}(v_k))-B,$$ for some finite $B$ independent of $k$, is also unbounded for $k\in\N$.

\noindent Suppose now $\tT$ and $\tT'$ are lifts of $T$ and $T-v$ so that $\bar{d}(\tilde{\Phi}_K^k(\tT),\tilde{\Phi}_K^k(\tT'))$ is bounded for $k\in\Z$ and $\tT'\ne \tT-v$. Since $\bar{d}(\tilde{\Phi}_K^k(\tT),\tilde{\Phi}_K^k(\tT-v))\to0$ as $k\to-\infty$, $\bar{d}(\tilde{\Phi}_K^k(\tT'),\tilde{\Phi}_K^k(\tT-v))$ must be bounded as $k\to-\infty$. But $\tT'$ and $\tT-v$ are not equal and both lie over $T-v$ so there is a nonzero $[h]\in H_1(X;\Z)/K$ with $\tT'+[h]=\tT-v$. Then the distance between $\tilde{\Phi}_K^k(\tT-v)=\tilde{\Phi}_K^k(\tT')+\bar{f}_*^k([h])$ and $\tilde{\Phi}_K^k(\tT')$ is certainly unbounded as $k\to-\infty$ since $\bar{f}_*^k([h])$, by hyperbolicity, takes on infinitly many different values in $H_1(X;\Z)/K$.
This establishes that $G$ is one-to-one on unstable manifolds and hence $G$ is globally finite-to-one.

That $G$ is $\mu$-a.e. $r$-to-1 (or $r'$-to-1) is a consequence of the ergodicity of $\Phi$ with respect to the unique invariant probability measure $\mu$ of the $\R^n$-action and the measurability of the function $f:\Omega\to\R$ given by $f(T):=\sharp G^{-1}(G(T))$.
\end{proof}

\section{Non-triviality of the global shadowing relation} 

To get some idea of what tilings are $\sim_{gsK}$-related, consider $K=K_{\Lambda}$:

\begin{prop}\label{K-equiv} Suppose that $\Phi$ is unimodular and hyperbolic and that $T,T'\in\Omega$ are $\Phi$-periodic and share a tile. Then $T\sim_{gsK_{\Lambda}}T'$.
\end{prop}
\begin{proof} For simplicity, assume that $T$ and $T'$ are fixed by $\Phi$. Let $K=K_{\Lambda}$ and suppose that $\tau\in T\cap T'$. Let $v\in int(\tau)$, so that $B_0[T-v]=B_0[T'-v]$. Choose $\tT=(T,\tx)\in\tO_{K}$ to be fixed by $\tilde{\Phi}_K$. There is then $\tx'\in\tX_{K}$ so that $\pi_2(\tT'-v)=\pi_2(\tT-v)$, where $\tT'=(T',\tx')$. Consider the loop $\alpha$ in the A-P complex X:
\[\alpha(t)=\left\{ \begin{array}{ll}
   p^T(2t\Lambda v+(1-2t)v) &  \mbox{if $0\le t\le1/2$},\\
   p^{T'}(2t\Lambda v+(1-2t) v) & \mbox{if $1/2\le t\le1.$}
 \end{array} \right.\]
 Then $l(\alpha)=0$ so the homology class of $\alpha$ is in $K$ and $\alpha$ lifts to a loop
 $\tilde{\alpha}$ in $\tX_K$ with $\tilde{\alpha}(0)=\pi_2(\tT-v)=\tilde{\alpha}(1)$. The lift $\tilde{\alpha}$ is of the form:
 \[\tilde{\alpha}(t)=\left\{ \begin{array}{ll}
   \tp_{\tx}^T(2t\Lambda v+(1-2t)v) &  \mbox{if $0\le t\le1/2$},\\
   \tp_{\ty}^{T'}(2t\Lambda v+(1-2t) v) & \mbox{if $1/2\le t\le1;$}
 \end{array} \right.\]
 where $\ty$ is determined by continuity. Since $\tilde{\alpha}(1)=\tp_{\ty}^{T'}(v)=\tilde{\alpha}(0)=
 \tp_{\tx}^T(v)=\tp_{\tx'}^{T'}(v)$, it must be the case that $\ty=\tx'$. Now $\pi_2(\tT-v)=\pi_2(\tT'-v)\Rightarrow \pi_2(\tT-\Lambda v)=\pi_2(\tilde{\Phi}_K(\tT-v))=\pi_2(\tilde{\Phi}_K(\tT'-v))=\pi_2(\tilde{\Phi}_K(\tT')-\Lambda v)$. We have $\pi_2(\tT'-\Lambda v)=\tp_{\tx'}^{T'}(\Lambda v)=\tilde{\alpha}(1/2)=\pi_2(\tT-\Lambda v)=\pi_2(\tilde{\Phi}_K(\tT')-\Lambda v)$ and hence $\tilde{\Phi}_K(\tT')=\tT'$. Since $\tT$ and $\tT'$ are both fixed by $\tilde{\Phi}_K$, $T\sim_{gsK}T'$.
\end{proof}

It is proved in \cite{BO} that, if $\Phi$ is an $n$-dimensional self-similar substitution (that is, $\Lambda=\lambda I$), there are $\Phi$-periodic $T\ne T'\in\Omega$ that agree in a half-space: there is 
$0\ne u\in\R^n$ and $R$ so that $B_0[T-v]=B_0[T'-v]$ for all $v\in\R^n$ with $\langle u,v\rangle\ge R$. For such $T,T'$, $\{T,T'\}$ is called an {\em asymptotic pair}.

\begin{cor}\label{asymptotic pairs} Suppose that $\Phi$ is self-similar. There are then $\Phi$-periodic $T\ne T'\in\Omega$ and $0\ne u\in\R^n$ so that $T\sim_{gsK_{\Lambda}} T'$ and $T-v\sim_{gsK_{\Lambda}} T'-v$ for all $v\in\R^n$ with $\langle u,v\rangle>0$.
\end{cor}
\begin{proof}
Let $T,T'$ be an asymptotic pair with direction vector $u$, as above. Then $T\sim_{gsK_{\Lambda}} T'$ by Proposition \ref{K-equiv}. There are then lifts $\tT,\tT'$ that are both fixed by $\tilde{\Phi}_K^m$ for some $m>0$. Then $\tT-v$ and $\tT'-v$ are lifts of $T-v$ and $T'-v$ with the properties: $\tilde{\Phi}_K^{km}(\tT-v)\to \tT$ and $\tilde{\Phi}_K^{km}(\tT'-v)\to \tT'$ as $k\to-\infty$. Also, if $\langle u,v\rangle>0$, there is $k'$ so that $\langle u,\lambda^{k'}v\rangle\ge R$, and it follows that $d(\Phi^{mk}(T-v),\Phi^{mk}(T'-v))\to 0$ as $k\to\infty$. We saw in the proof of Proposition \ref{K-equiv} that if 
$\tT$ is fixed by $\tilde{\Phi}_K^{m}$, $T'$ is fixed by $\Phi^m$, and $\pi_2(\tT-w)=\pi_2(\tT'-w)$, then $\tT'$ is also fixed by $\tilde{\Phi}_K^m$. It follows from hyperbolicity of $\tilde{\Phi}_K$ on deck transformations that, conversely, if $\tT$ and $\tT'$ are both fixed by $\tilde{\Phi}_K$ and $B_0[T-w]=B_0[T'-w]$ (so that $p^T(w)=p^{T'}(w)$) then $\pi_2(\tT-w)=\pi_2(\tT'-w)$. If $\langle u,v\rangle>0$, then $B_0[T-w]=B_0[T'-w]$ for $w=\lambda^k v$ and $k>k'$. Consequently, for such $v$, $\bar{d}(\tilde{\Phi}_K^{mk}(\tT-v),\tilde{\Phi}_K^{mk}(\tT'-v))=d(\Phi^{mk}(T-v),\Phi^{mk}(T'-v))+
\td(\pi_2(\tT-\lambda^{mk}v),\pi_2(\tT'-\lambda^{mk}v))\to 0$ as $k\to\infty$, so $T-v\sim_{gsK_{\Lambda}}T'-v$.
\end{proof}

\section{Global shadowing and regional proximality for Pisot family substitutions}\label{global shadowing}

A fundamental result of Auslander (\cite{Aus}) asserts that the structure relation of the maximal equicontinuous factor for an abelian group action on a compact metrizable space is given by regional proximality. In the context of the $\R^n$-action on a tiling space $\Omega$, two tilings
$T,T'\in\Omega$ are {\bf regionally proximal} provided, for every $\epsilon>0$, there are $S,S'\in\Omega$
and $v\in\R^n$ so that: (i) $d(T,S)<\epsilon$; (ii) $d(T',S')<\epsilon$; and (iii) $d(S-v,S'-v)<\epsilon$.
If $T$ and $T'$ are regionally proximal, we will write $T\sim_{rp}T'$. Thus $T\sim_{rp}T'$ if and only if $g(T)=g(T')$, where $g$ is the factor map onto the maximal equicontinuous factor of the translation action on $\Omega$.

We thus have two closed equivalence relations on a substitution tiling space $\Omega_{\Phi}$: regional proximality and global shadowing. The aim of this section is to compare these two relations.
The following Proposition \ref{gs implies rp} shows that for unimodular hyperbolic substitutions, global shadowing is stronger than regional proximality. Proposition \ref{rp implies gs} establishes that if the substitution is Pisot family, and $D=D(\Lambda)=D(GR)$, then the regional proximal and global shadowing relations coincide. From this we will easily deduce Theorem \ref{gs=rp}, and Corollaries \ref{g finite-to-one} and \ref{hypcon implies Pcon}.

First, a lemma about regional proximality.

\begin{lemma}\label{rp} Suppose that $T,S\in\OP$, $k_i\to\infty$, and $v_i\in\R^n$ are such that
$\Phi^{k_i}(T)\to\bar{T}\in\OP$, $\Phi^{k_i}(S)\to\bar{S}\in\OP$, and $d(\Phi^{k_i}(T-v_i),\Phi^{k_i}(S-v_i))\to0$. Then $\bar{T}\sim_{rp}\bar{S}$.
\end{lemma}
\begin{proof}
Let $g$ be the map of $\OP$ onto the maximal equicontinuous factor $\OP/\sim_{rp}$. If $d(\Phi^{k_i}(T-v_i),\Phi^{k_i}(S-v_i))=d(\Phi^{k_i}(T)-\Lambda^{k_i}v_i,\Phi^{k_i}(S)-\Lambda^{k_i}v_i)\to0$, then
$d(g(\Phi^{k_i}(T))-\Lambda^{k_i}v_i,g(\Phi^{k_i}(S))-\Lambda^{k_i}v_i)\to0$ by uniform continuity of $g$. Then, by equicontinuity of the $\R^n$-action on $\OP/\sim_{rp}$, $d(g(\Phi^{k_i}(T)),g(\Phi^{k_i}(S)))\to0$. Thus $d(g(\bar{T}),g(\bar{S}))=0$ and $\bar{T}\sim_{rp}\bar{S}$.
\end{proof}

Given $\tx,\tx'\in\tX=\tX_{K_{\Lambda}}$ and a path $\tilde{\gamma}$ in $\tX$ from $\tx$ to $\tx'$, let 
$l(\tilde{\gamma}):=l(\gamma)$, where $\gamma$ is the path $\gamma=\pi\circ\tilde{\gamma}$ in $X$ (see Lemma \ref{return vector}). If $\tilde{\gamma}'$ is any other such path, then the concatenation of $\tilde{\gamma}'$ with the reverse of $\tilde{\gamma}$ is a loop in $\tX$ and hence must project to an element of $K_{\Lambda}$, that is, to a loop of displacement 0. Thus the displacement $l(\tilde{\gamma})$ depends
only on $\tx$ and $\tx'$, and not on $\tilde{\gamma}$.

\begin{Proposition}\label{gs implies rp} Suppose that $\Phi$ is unimodular and hyperbolic. Then the
global shadowing relation is contained in the regional proximality relation: $T\sim_{gsK_{\Lambda}}T'\implies T\sim_{rp}T'$.
\end{Proposition}
\begin{proof}
We suppose that the A-P complex $X$ is collared. In this situation, if $S,S'\in\OP$ are such that $p(S)=p(S')$, then $d(\Phi^k(S),\Phi^k(S'))\to0$ as $k\to\infty$.

Suppose that $\tT=(T,\tx),\tT'=(T',\tx')\in\tilde{\Omega}$ are such that $\bar{d}(\tilde{\Phi}^{k}(\tT),\tilde{\Phi}^{k}(\tT'))$, $k\in\Z$, is bounded. There are then: a sequence $k_i\to-\infty$; $[h_i]$ in the group
$H_1(X;\Z)/K_{\Lambda}$ of deck transformations of $\tX$; and $n$-cells $\tilde{\tau},\tilde{\tau}'$ of $\tX$ so that $\tx_i\in[ h_i](\tilde{\tau})$ and $\tx'_i\in [h_i](\tilde{\tau}')$ for all $i$, where $(\tT_i,\tx_i):=
\tilde{\Phi}^{k_i}(\tT)$ and $(\tT'_i,\tx'_i):=
\tilde{\Phi}^{k_i}(\tT')$. Let $\tilde{\gamma}^i$ be a path in $\tX$ from $\tx_i$ to $\tx'_i$. We may take $\tilde{\gamma}^i=\tilde{\gamma}^i_1*\cdots*\tilde{\gamma}^i_l$ with each $\tilde{\gamma}^i_j$ a lift of a path $\gamma^i_j$ in $X$ of the form
$\gamma^i_j(t)=p(S^i_j-tv^i_j)$, $t\in[0,1]$, for some $S^i_j\in\OP$ and $v^i_j\in\R^n$ with $S^i_1=T_i$ and $S^i_l=T'_i $. (It is key to this argument that $l$ is constant, independent of $i$.) 
 We have $p(S^i_j-v^i_j)=p(S^i_{j+1})$ for $j=1,\ldots,l-1$ and all $i$. Passing to a subsequence, we may assume that $\Phi^{|k_i|}(S^i_j+v^i_1+\cdots+v^i_{j-1})\to\bar{S}_j\in\OP$ as $i\to\infty$ for $j=2,\ldots,l$. Let $\bar{S_1}=T$. We conclude from Lemma \ref{rp} that $\bar{S}_j\sim_{rp}\bar{S}_{j+1}$ for $j=1\ldots,l-1$. 
 
 Now, since $\tilde{\Phi}^{|k_i|}\circ\tilde{\gamma}_i$ is a path from $\tx$ to $\tx'$, the vector $\Lambda^{|k_i|}(v^i_1+\cdots+v^i_{l-1})=:v$ is constant (this is the displacement vector $l(\tilde{\Phi}^{|k_i|}\circ\tilde{\gamma}_i)$ of the path $\tilde{\Phi}^{|k_i|}\circ\tilde{\gamma}_i$ in $\tX$). Thus $\bar{S}_l=T'+v$. We have a path in $\tX$, call it $\tilde{\gamma}$, from $\tx$ to $\tx'$ with displacement $v$. Then  any path in $\tX$ from $\tf^k(\tx)$ to $\tf^k(\tx')$ has displacement $\Lambda^kv$: since $\bar{d}(\tilde{\Phi}^k(\tT),\tilde{\Phi}^k(\tT'))$ is bounded, and $\tilde{\Phi}^k(\tT)=(\Phi^k(T),\tf^k(\tx))$ and  $\tilde{\Phi}^k(\tT')=(\Phi^k(T'),\tf^k(\tx'))$, $v$ must be zero - that is, $\bar{S}_l=T'$. Since $\sim_{rp}$ is a translation invariant equivalence relation, $T\sim_{rp}T'$.
\end{proof}

\begin{Proposition}\label{rp implies gs}  Suppose that $\Phi$ is a unimodular Pisot family $n$-dimensional substitution with linear expansion $\Lambda$. If $D(\Lambda)=D(GR)=D$, then the global shadowing relation contains the regional proximal relation: $T\sim_{rp}T'\implies T\sim_{gsK_{\Lambda}}T'$.
\end{Proposition}
\begin{proof}
Let $f:X\to X$ be the substitution induced map on the collared A-P complex for $\Phi$. Let $A$ represent $\bar{f}_*:H_1(X;\Z)/K_{\Lambda}\to H_1(X;\Z)/K_{\Lambda}$
in some basis, say $\{[h_1],\ldots,[h_D]\}$, which we fix for the remainder of this proof. $A$ is then hyperbolic, and, as an isomorphism of $\R^D$, has invariant stable and unstable spaces $E^s$ and $E^u$ with $E^s\oplus E^u=\R^D$. There are $\eta\in (0,1)$ and $C$ so that $|A^kx|\le C\eta^k|x|$ for $x\in E^s$, $k\in\N$ and $|A^{-k}x|\le C\eta^k|x|$ for $x\in E^u$, $k\in\N$. Let $x=x^s+x^u$ be the decomposition of $x\in\R^D$ into stable and unstable parts. Thus, for each $[h]\in H_1(X;\Z)/K_{\Lambda}$, there are corresponding $[h]^s\in E^s,[h]^u\in E^u$. 

Given $T\in\OP$ and $\tx\in\tX:=\tX_{K_{\Lambda}}$ with $\pi(\tx)=p(T)$, we will call the subset $\mathcal{S}(T,\tx):=\{\tilde{p}^T_{\tx}(v):v\in\R^n\}$ of $\tX$ a {\em sheet}.

\begin{claim}\label{sheets} There is $B\in\R$ so that if $\mathcal{S}$ is any sheet in $\tX$ and $[h]\in H_1(X;\Z)/K_{\Lambda}$ is such that $[h](\mathcal{S})\cap\mathcal{S}$ contains an $n$-cell of $\tX$, then $|[h]^s|\le B$.
\end{claim}
To prove the claim, let $m\in\N$ and $\alpha<1$ be so that $|A^{mk}x|\le \alpha^k|x|$ for all $k\in\N$ and $x\in E^s$. We may assume that $m$ is large enough so that, for each tile $\tau$, $\Phi^m(\tau)$ contains, in its interior, a tile of every type. Let $B_1$ be large enough so that if $\tilde{\tau}$ is any $n$-cell in $\tX$, $\tilde{\tau}_1,\tilde{\tau}_2$ are two $n$-cells in $\tf^{2m}(\tilde{\tau})$ of the same type,  and $[h](\tilde{\tau}_1)=\tilde{\tau}_2$, then
$|[h]^s|\le B_1$. Now let $T\in\OP$ be fixed by $\Phi^m$ and let $\tx\in\tX$ be fixed by $\tf^m$, with
$\pi(\tx)=p(T)$. The sheet $\mathcal{S}:=\{\tilde{p}^T_{\tx}(v):v\in\R^n\}$ is invariant under $\tf^m$.  We may assume that $\tx$ is in the interior of an $n$-cell $\tilde{\tau}$ of $\tX$ so that $\mathcal{S}=\cup_{k\in\N}\tilde{\Phi}^{km}(\tilde{\tau})$. Let $B_2$ be large enough so that $\alpha B_2+B_1\le B_2$. 

{\bf Subclaim:} If $[h](\tx)\in\mathcal{S}$ then $|[h]^s|\le B_2$.

To see that this is the case, let $\mathcal{S}_k:=\cup_{j=0,\ldots,k}\tf^{mj}(\tilde{\tau})$ for $k=0,1,\ldots$. We prove the claim, with $\mathcal{S}$ replaced by $\mathcal{S}_k$, by induction on $k$.
If $[h](\tx)\in\mathcal{S}_0$, then $[h]^s=0$ and if $[h](\tx)\in\mathcal{S}_1$, then $|[h]^s|\le B_1<B_2$. Suppose the claim is true with $\mathcal{S}$ replaced with $\mathcal{S}_k$, $k\ge 1$, and suppose that $[h](\tx)\in\mathcal{S}_{k+1}$. Let $\tilde{\tau}'$ be the $n$-cell of $\mathcal{S}_k$ with $[h](\tx)\in\tf^{2m}(\tilde{\tau}')$. There is then an $n$-cell of the same type as $\tilde{\tau}$ in $\tf^m(\tilde{\tau}')\subset\mathcal{S}_k$. Thus there is $[h_1]$ with $[h_1](\tx)\in\tf^m(\tilde{\tau}')$: by hypothesis, $|[h_1]^s|\le B_2$. Then $\tf^m([h_1](\tx))=[f^m_*(h_1)](\tx)\in\tf^{2m}(\tilde{\tau}')$, so that $|[h]^s-[f^m_*(h_1)]^s|=|[h-f^m_*(h_1)]^s|\le B_1$. Also, $|[f^m_*(h_1)]^s|\le\alpha|[h_1]^s|\le\alpha B_2$, so that $|[h]^s|\le \alpha B_2+B_1\le B_2$, completing the inductive argument and establishing the subclaim.

For each of the finitely many different types of $n$-cell in $\tX$, represented, say, by $\tilde{\tau}_1,\ldots,\tilde{\tau}_r$, there is thus a sheet $\mathcal{S}_i=\mathcal{S}(T_i,\tx_i)$ so that if the $n$-cell $\tilde{\tau}\subset
\mathcal{S}_i$ has the same type as $\tilde{\tau}_i$, and $[h](\tilde{\tau})\subset\mathcal{S}_i$, then $|[h]^s|\le 2B_2:=B$. Now if $\mathcal{S}=\mathcal{S}(T,\tx)$ is any sheet with $\tilde{\tau}\subset \mathcal{S}\cap[h](\mathcal{S})$, let $\tilde{\tau}_i$ be of the same type as $\tilde{\tau}$. 
Let $Q$ be a connected finite patch of $T$ containing the tiles $\tau$ and $\tau-l([h])$, with
$\tilde{p}^T_{\tx}(\tau)=\tilde{\tau}$, and hence also $\tilde{p}^T_{\tx}(\tau-l([h]))=[-h](\tilde{\tau})$.
There is then $v$ with $Q-v\subset T_i$. Then $[h](\tilde{p}^T_{\tx_i}(\tau-l[h]-v)=\tilde{p}^T_{\tx_i}(\tau-v)$, so $|[h]^s|\le B$. This establishes Claim \ref{sheets}.

Suppose that $T,T'\in\OP$ are such that $T\sim_{rp}T'$. For each $r=1,2,\ldots$ there are $S_r,S'_r\in\OP$ and $v_r,w_r\in\R^n$ so that: $B_r[T]=B_r[S_r]$; $B_r[S_r-v_r]=B_r[S'_r-v_r]$; $B_r[T'-w_r]=B_r[S'_r]$; and $w_r\to0$ as $r\to\infty$. Pick a lift $\tT=(T,\tx)\in\tilde{\Omega}$, then choose lifts
$\tilde{S}_r=(S_r,\tx),\tilde{S}'_r=(S'_r,\tx'_r)\in\tilde{\Omega}$ so that $\tilde{p}^{S_r}_{\tx}(v_r)=\tilde{p}^{S'_r}_{\tx'_r}(v_r)$. 

\begin{claim}\label{large r} For sufficiently large $r$, $\tx'_r=:\tx'$ is constant and $w_r=0$.
\end{claim}
To see this, let $R$ be large enough so that every $R$-patch contains a tile of every type. Let $\tilde{\tau}$ be an $n$-cell of $\tX$ with $\tx\in\tilde{\tau}\subset \mathcal{S}(T,\tx)$ and for each 
$r\ge R$, let $\tilde{\tau}_r$ be an $n$-cell of $\tX$ with $\tilde{\tau}_r\subset \mathcal{S}(S_r,\tx)\cap\mathcal{S}(S'_r,\tx'_r)$ of the same type as $\tilde{\tau}$. Pick also $n$-cells $\tilde{\tau}_r'\subset \mathcal{S}(S'_r,\tx')$, of the same type as $\tilde{\tau}$, in $\tilde{p}^{S'_r}_{\tx'_r}(B_R(0))$. As $B_r[S'_r]=B_r[T'-w_r]$, we may take $\tilde{\tau}'_r=\tilde{p}^{S'_r}_{\tx'_r}(\tau'-w_r)$ with $\tau'\in T'$. 
Let $L:\R^D\to\R^n$ be the linear map given by $L((a_1,\ldots,a_D)):=\sum_{i=1}^Da_il([h_i])$. Then $LA=\Lambda L$ and, since $D=D(\Lambda)$, $L|_{E^u}:E^u\to\R^n$ is an isomorphism. 
Also, $L$ is one-to-one on $\Z^D$ (recall that $l:H_1(X;\Z)\to GR$ has kernel $K_{\Lambda}$).
Thus $\Gamma:=L(\{a\in\R^D:|a^s|\le B\})$, $B$ as in Claim \ref{sheets}, is a regular model set. In particular, $\Gamma$ is uniformly discrete, relatively dense, and has the Meyer property (see \cite{M}). Let $[h^1_r](\tilde{\tau})=\tilde{\tau}_r$ and $[h^2_r](\tilde{\tau}'_r)=\tilde{\tau}_r$. Then
$|[h^1_r]^s|\le B$ and $|[h^2_r]^s|\le B$ by Claim \ref{sheets}, so $l([h^1_r]),l([h^2_r])\in\Gamma$ and $\{l([h^1_r])-l([h^2_r]): r\ge R\}$ is uniformly discrete. But $(l([h^1_{r'}])-l([h^2_{r'}]))-(l([h^1_r])-l(h^2_r]))=w_r-w_{r'}$, and since $w_r\to 0$, it must be that $w_r=0$ for all sufficiently large $r$.
Thus, $l([h^1_r-h^2_r])$ is constant for large $r$, and since $l$ is injective on $H_1(X;\Z)$, $[h^1_r-h^2_r]:=[h]$ is constant for large $r$. Thus the $n$-cells $\tilde{\tau}'_r=[h](\tilde{\tau})$ and the patches $B_R[S'_r]$ are constant for all large $r$ and it follows that $\tx'_r$ is also constant for large $r$, establishing Claim \ref{large r}.

We can deduce more from the above. For large $r$ let $\tilde{\gamma}_1(t):=\tilde{p}^{S_r}_{\tx}(tv_r)$, $\tilde{\gamma}_2(t):=\tilde{p}^{S'_r}_{\tx'}((1-t)v_r)$, and $\tilde{\gamma}_3(t):=\tilde{p}^{T'}_{\tx'}(tw)$, $0\le t\le1$, where $w$ is such that $\tilde{p}^{T'}_{\tx'}(w)=[h](\tx)$. Then $|l([h])|=|l(\tilde{\gamma}_1*\tilde{\gamma}_2*\tilde{\gamma}_3)|=|v_r-v_r+w|\le R$. Let $\bar{R}_1:=sup\{\td(\ty,[h'](\ty)):\ty\in\tX,[h']\in\Gamma, |[h']|\le R\}$. Then $\bar{R}_1<\infty$ since $\Gamma$ is a discrete subset of $\R^n$ and the metric $\td$ is equivariant with respect to the deck transformations. Let $\bar{R}_2:=sup\{\td(\tilde{p}^S_{\ty}(w),\ty):|w|\le R, (S,\ty)\in\tilde{\Omega}\}$. $\bar{R}_2$ is also finite by equivariance. We have: $\td(\tx,\tx')\le \bar{R}_1+\bar{R}_2$.

Now let $\tT'=(T',\tx')$. For $k\in\N$, it is clear that the sheets $\tf^k(\mathcal{S}(T,\tx))=\mathcal{S}(\Phi^k(T),\tf^k(\tx))$, $\tf^k(\mathcal{S}(S_r,\tx))=\mathcal{S}(\Phi^k(S_r),\tf^k(\tx))$, $\tf^k(\mathcal{S}(S'_r,\tx'))=\mathcal{S}(\Phi^k(S'_r),\tf^k(\tx'))$, and $\tf^k(\mathcal{S}(T',\tx'))=\mathcal{S}(\Phi^k(T'),\tf^k(\tx'))$ overlap in the same manner as those sheets with $k=0$, and the above argument 
yields $\td(\tf^k(\tx),\tf^k(\tx'))\le \bar{R}_1+\bar{R}_2$. Now fix $k<0$.  Given $r$, there is $r'=r'(r)$ big enough
so that if $S,S'\in\OP$ satisfy $B_{r'}[S]=B_{r'}[S']$, then $B_r[\Phi^k(S)]=B_r[\Phi^k(S')]$. (This consequence of the invertibility of $\Phi$ is referred to as ``recognizability", see, for example, \cite{sol}.) For $r>0$ and $r'=r'(r)$, consider $\tilde{\Phi}^k((T,\tx))=(\Phi^k(T),\ty)$ and $\tilde{\Phi}^k((S_{r'},\tx))=(\Phi^k(S_{r'}),\ty')$. Since $B_0[\Phi^k(S_{r'})]=B_0[\Phi^k(T)]$, $\pi(\ty)=\pi(\ty')$ and there is $[h]\in H_1(X;\Z)/K_{\Lambda}$ so that $[h](\ty')=\ty$. Then $\tx=\tf^k(\ty')=\tf^k([h](\ty))=
\bar{f}_*^k([h])(\tf^k(\ty))=\bar{f}_*^k([h])(\tx)$, so $[h]=0$ ($\bar{f}_*$ is invertible on $H_1(X;\Z)/K_{\Lambda}$).
Thus $\ty'=\ty$. In this way, we see that the sheets determined by $\tilde{\Phi}^k(\tT)$, $\tilde{\Phi}^k(\tilde{S}_{r'})$, $\tilde{\Phi}^k(\tilde{S}'_{r'})$, and $\tilde{\Phi}^k(\tT')$, for $r'$ large, overlap as above so that $\bar{d}(\tilde{\Phi}^k(\tT)'\tilde{\Phi}^k(\tT'))\le d(\Phi^k(T),\Phi^k(T'))+\bar{R}_1+\bar{R}_2$. We conclude that $\bar{d}(\tilde{\Phi}^k(\tT),\tilde{\Phi}^k(\tT'))$, $k\in\Z$, is bounded; that is, $T\sim_{gsK_{\Lambda}}T'$.

\end{proof}

\begin{proof} (Of Theorem \ref{gs=rp}.)
By Proposition \ref{main prop}, $G(T)=G(T')$ if and only if $T\sim_{gsK_{\Lambda}}T'$, and, by the fundamental result of Auslander (\cite{Aus}), $g(T)=g(T')$ if and only if $T\sim_{rp}T'$. By Propositions \ref{gs implies rp} and \ref{rp implies gs}, $T\sim_{gsK_{\Lambda}}T'$ if and only if $T\sim_{rp}T'$. Thus we may identify $G$ with $g$.
\end{proof}

\begin{proof} (Of Corollary \ref{g finite-to-one}.)
The first statement follows immediately from Theorems \ref{main theorem} and \ref{gs=rp} and the characterization of $cr$ in \cite{BK}. For the second statement: the eigenvalues form subgroup of $\R^n$ so are relatively dense if and only if their linear span is all of $\R^n$. Suppose there is a vector $v\ne0$ perpendicular to the linear span of the eigenvalues $E$. If $f:\OP\to\T$ is a continuous eigenfunction with eigenvalue $\beta$ and $T\in\OP$, then $f(T-tv)=exp(2\pi i\langle \beta,tv\rangle)f(T)=f(T)$ for all $t\in\R$. This means that $g(T-tv)=g(T)$ for all $t\in\R$ (the Halmos - von Neumann theory asserts that the map $g$ is determined by the continuous eigenfunctions - see \cite{R} or \cite{BK}). But then $g$ is not finite-to-one. 

It is shown in \cite{LS2} that, in this context, the third statement is equivalent to the second. 

\end{proof}

\begin{proof} (Of Corollary \ref{hypcon implies Pcon}.)
The spectrum of the $\R^n$-action on $\OP$ is pure discrete if and only if $g$ is a.e. one-to-one (see, for example, \cite{BK}).
\end{proof}

\section{Pisot family substitutions}\label{Pisot family substitutions}

\begin{proof}(of Theorem \ref{Pisot}) Suppose that $\Phi$ is unimodular of $(m,d)$-Pisot family type. Let $X$ be the collared A-P complex for $\Phi$ so that $\il f$ is identified with $\Omega$ via $\hat{p}$. We prove that, for $T,T'\in \Omega$, $g(T)=g(T')$ if and only if $G(T)=G(T')$. 

Suppose first that $G(T)=G(T')$. 
Let $\tX=\tX_{K_{\Lambda}}$ and $\tf=\tf_{K_{\Lambda}}:\tX\to\tX$. We will argue that $g:\Omega=\il f\to\T^{md}$ lifts to $\tilde{g}:\tO_{K_\Lambda}=\il\tf\to\R^{md}$. We may suppose that the origin is in the interior of each prototile. For each $k\in\N$, let $X_k$ denote the A-P complex made of $k$-th order supertiles. That is, the $n$-faces of $X_k$ are the patches $\Phi^k(\rho_i)$, with face $\Phi^k(\rho_i)$ glued to face $\Phi^k(\rho_j)$ along $(n-1)$-face $\Lambda^ke$ in $X_k$ if and only if the prototile $\rho_i$ is glued to the prototile $\rho_j$ along the $(n-1)$-face $e$ in $X$. As each tiling in $\Omega$ is uniquely tiled by 
$k$-th order super tiles, there is a natural map $p^k:\Omega\to X_k$ ($p^k(T)=[v]$ in the face $\Phi^k(\rho_j)$ of $X_k$ if the $k$-th order super tile of $T$ containing the origin is $\Phi^k(\rho_j)-v$). The decomposition of $k$-th order 
supertiles into $(k-1)$-st order supertiles induces maps $f_k:X_k\to X_{k-1}$ so that $\Omega\simeq\il f_k$. Furthermore, there are substitution induced maps $f^k:X_k\to X_k$ with $f_k\circ f^k=f^{k-1}\circ f_k$ so that $\Phi$ on $\Omega$ is conjugated with the homeomorphism induced by $(f^k)$ on $\il f_k$ by $(p^k)$. In order to lift $g$ we approximate $g$ by maps $g_k\circ p^k$, where the 
$g_k:X_k\to\T^{md}$ are constructed as follows. For each $k$-th order super tile $\rho^k:=\Lambda^k(\rho)$, $\rho$ a prototile for $\Phi$, and $\epsilon_k>0$, let $\mathcal{N}(\rho^k,\epsilon_k)$ be the $\epsilon_k$-neighborhood of the boundary of $\rho^k$. Choose (arbitrarily ) $T_{\rho^k}\in\Omega$ so that the supertile $\rho^k$ occurs in the decomposition of $T_{\rho^k}$ into $k$-th order supertiles (we mean this supertile occurs exactly, with 0 translation). Define $g_k$ on $\rho^k\setminus\mathcal{N}(\rho^k,\epsilon_k)$ by $g_k(p^k(T_{\rho^k}-v)):=g(T_{\rho^k}-v)$ for $v\in \rho^k\setminus \mathcal{N}(\rho^k,\epsilon_k)$. Since we have collared $\Phi$, there is $\delta=\delta(k,\epsilon_k)$ so that if $\rho^k_i$ and $\rho^k_j$ are adjacent faces in $X_k$, glued along an $l$-face $e$, and $e\ni x=p^k(T_{\rho^k_i}-v)=p^k(T_{\rho^k_j}-w)$, with $v\in\rho^k_i$ and $w\in \rho^k_j$, then $d(T_{\rho^k_i}-v,T_{\rho^k_j}-w)<\delta$, and $\delta\to0$ as $k\to\infty$. Choosing $\epsilon_k$ sufficiently small, and using the local convexity of $\T^{md}$, we extend $g_k$ continuously to all of $X_k$ so that, for each $i$, if $v,w\in \rho^k_i$ are such that $|v-w|<\epsilon_k$, then $d(g_k(p^k(T_{\rho^k_i}-v)),g_k(p^k(T_{\rho^k_i}-w)))<\delta$. 
Now for $T\in\Omega$, let $T_{\rho^k_i}$ and $v\in\rho^k_i$ be such that $p^k(T)=p^k(T_{\rho^k_i}-v)$. Since we have collared, and $0\in int(\rho_i)$, $T$ and $T_{\rho^k_i}-v$ agree in an $r_k$-ball about the origin, with $r_k\to\infty$ as $k\to\infty$. Thus $d(g(T),g_k\circ p^k(T))=
d(g(T),g_k\circ p^k(T_{\rho^k_i}-v))\le d(g(T),g(T_{\rho^k_i}-v))+d(g(T_{\rho^k_i}-v,g_k\circ p^k(T_{\rho^k_i}-v))\to0$ as $k\to\infty$. That is, $g_k\circ p^k\to g$ uniformly as $k\to\infty$.

With the goal still of lifting $g$, we show now that $g_k:X_k\to\T^{md}$ lifts to $\tilde{g}_k:\tX_k\to\R^{md}$. Let $E_k$ be a dual 1-skeleton in $X_k$. That is, we choose a vertex in the interior of each $n$-face, a vertex in the (relative) interior of each $(n-1)$-face of $X_k$, and we make a straight line edge from the vertex interior to every $n$-face to each of the vertices on the $(n-1)$ -subfaces. Let $\tilde{\gamma}:[0,1]\to\tX_k$ be a loop in $\tX_k$. Then the loop $\pi_k\circ\tilde{\gamma}$ in $X_k$ is homotopic to a piecewise affine loop $\gamma$ lying on $E_k$. We can write $\gamma$ as a concatenation $\gamma=c_1*c_2*\cdots* c_l$ with each $c_i$ of the form
$c_i(t)=p^k(T_{\rho^k_j}-(v_i+tw_i))$, $t_i\le t\le t_{i+1}$, for some $j=j(i),v_i,w_i$. Let $\eta:=g_k\circ\gamma:[0,1]\to\T^{md}$ and let $\tilde{\eta}:[0,1]\to\R^{md}$ be a lift of $\eta$. Let $\iota:\R^n\to\R^{md}$ be the linear embedding so that $g(T-v)=g(T)-\iota(v)$. We have $$\tilde{\eta}(1)-\tilde{\eta}(0)=\sum_{i=0}^{l-1}(\tilde{\eta}(t_{i+1})-\tilde{\eta}(t_i))\approx\sum_{i=0}^{l-1}(\widetilde{\iota(v_{j(i)})}+t_{i+1}\iota(w_i))-(\widetilde{\iota(v_{j(i)})}+t_{i}\iota(w_i))$$ $$=\sum_{i=0}^{l-1}(t_{i+1}-t_i)\iota(w_i),$$ with the approximation improving as $k$ gets larger. The covering space $\tX_k$ is the quotient of the abelian cover $(\tX_{k})_{ab}$ by the action of those deck transformations lying in the kernel $K_k$ of the homomorphism 
$l_k:H_1(X_k;\Z)\to\R^n$. That $\tilde{\gamma}$ is a loop in $\tX_k$ means that the homology class
$[\gamma]$ lies in $K_k$. That is, $l_k([\gamma])=\sum_{i=0}^{l-1}(t_{i+1}-t_i)w_i=0$. Thus $\sum_{i=0}^{l-1}(t_{i+1}-t_i)\iota(w_i)=0$ also. As $\tilde{\eta}(1)-\tilde{\eta}(0)\in\Z^{md}$, it must be that, for sufficiently large $k$, $\tilde{\eta}(1)-\tilde{\eta}(0)=0$. The lifting criterion is satisfied: $g_k\circ\tilde{\pi}\circ\tilde{\gamma}$ is homotopic to a loop pushed down from $\R^{md}$ to$\T^{md}$. 

Let us record a property of $\tilde{g}_k$ for later use. From the displayed approximation of 
$\tilde{\eta}(1)-\tilde{\eta}(0)$ above, we see that, for sufficiently large k and any loop $\gamma$ in
$X_k$, and $\eta:=g_k\circ\gamma$, $\tilde{\eta}(1)-\tilde{\eta}(0)=\iota\circ l_k([\gamma])$. It follows that for $\tx\in \tX_k$, $[h]\in H_1(X_k;\Z)/K_k$, and $k$ sufficiently large:
\begin{equation}\label{g} 
\tilde{g}_k(\tx+[h])=\tilde{g}_k(\tx)+\iota\circ l_k(h).
\end{equation}

In the constructions above we have identified $\Omega$ with $\il f$ (by means of $\hat{p}:T\mapsto
(p(\Phi^{-i}(T)))$) and $\il f$ with $\il (f_k)$ by means of a rescaling: since $\rho^{k+1}_i=\Lambda \rho^k_i$ and $X=X_0$ there are natural maps $\Lambda^k:X\to X_k$ with 
$\Lambda^k\circ f=f_k\circ\Lambda^{k+1}$. Then $\il f$ is identified with $\il (f_k)$ by $(x_i)\mapsto
(\Lambda^i(x_i))$. The spaces $\il\tf$ and $\il (\tf_k)$ are then identified by first choosing a lift of $f$ to $\tf$, then choosing lifts $(\tilde{\Lambda^k}),(\tilde{f}_k)$ so that $\tilde{\Lambda^k}\circ\tf=\tf_k\circ\widetilde{\Lambda^{k+1}}$. Now choose lifts $\tilde{g}_k:\tX_k\to\R^{md}$ so that 
$|\tilde{g}_k\circ\tf_{k+1}(\tx_{k})-\tilde{g}_{k+1}(\tx_k)|\to0$ for some (and hence any) $(\tx_k)\in\il(\tf_k)$ and let $\tilde{g}:\il\tf\to\R^{md}$ be defined by $\tilde{g}((\tx_i)):=\lim_{k\to\infty}\tilde{g}_k\circ\tilde{\Lambda^k}(\tx_k)$. Convergence of this limit to a lift of $g$ is assured by the convergence of $g_k\circ p^k$ to $g$.

It is proved in \cite{BK} that there is a hyperbolic and unimodular $md\times md$ integer matrix $B$
so that $g\circ\Phi=F_B\circ g$, $F_B(x+\Z^{md}):=Bx+\Z^{md}$ being the corresponding hyperbolic 
automorphism of $\T^{md}$. It follows that $\tilde{g}\circ\hat{\tf}=B\tilde{g}$. We have assumed that $G(T)=G(T')$ with the objective of showing that then $g(T)=g(T')$. Let $(x_i),(x'_i)\in\il f$ be such that $\hat{p}(T)=(x_i),\hat{p}(T')=(x'_i)$. By Proposition \ref{main prop}, $(x_i)\sim_{gsK_{\Lambda}}(x'_i)$, so there are $(\tx_i),(\tx'_i)\in\il\tf$, living over $(x_i),(x'_i)$, so that $\bar{d}(\hat{\tf}^k((\tx_i)),\hat{\tf}^k((\tx'_i)))$, $k\in\Z$, is bounded. Then $|\tilde{g}((\hat{\tf}^k((\tx_i)))-\tilde{g}(\hat{\tf}^k((\tx'_i)))|=|B^k\tilde{g}((\tx_i))-B^k\tilde{g}((\tx'_i))|$, $k\in\Z$, is also bounded, and hyperbolicity of $B$ implies that $\tilde{g}((\tx_i))=\tilde{g}((\tx'_i))$. Thus $g(T)=g(T')$.

Now suppose that $g(T)=g(T')$. We wish to show that $T\sim_{gsK_{\Lambda}} T'$. From Theorem 3.1 of \cite{LS} we have (for $\Phi$ of $(m,d)$-Pisot family type) that there are $v_1,\ldots,v_m\in\R^n$ with $GR(\Phi)\subset \Z[\Lambda]v_1+\cdots+\Z[\Lambda]v_m$. Thus $D(GR)\le D(\Lambda)$, and hence $D(GR)=D(\Lambda)$ since the opposite inequality is always satisfied (Lemma \ref{eigenvalues}). Let  $K$ be the kernel of $l:H_1(X;\Z)\to GR(\Phi)$, let $\pi_K:\tX_K\to X$ be the 
corresponding cover of the collared A-P complex $X$, and let $\tf_K:\tX_K\to\tX_K$ be a lift of $f$. Let $\pi_1:\tilde{\Omega}\to\Omega$, with $\tilde{\Omega}:=\{(T,\tx):p(T)=\pi_K(\tx)\}$ the 
cover isomorphic with $\hat{\pi}_K:\il \hat{\tf}_K\to\il f$ as in Lemma \ref{tO}. 
Let $\bar{d}$ be the metric in $\tilde{\Omega}_K$ given by $\bar{d}((T,\tx),(T',\ty))=d(T,T')+\td(\tx,\ty)$. We will show that there are $\tT,\tT'\in\tilde{\Omega}_K$ lying over $T,T'$ so that $\bar{d}(\tilde{\Phi}^k(\tT),\tilde{\Phi}^k(\tT'))$, $k\in\Z$, is bounded. 

It is proved in \cite{BK} that $g(T)=g(T')$ if and only if $T$ and $T'$ are {\em strongly regionally proximal}, i.e., for each $k\in\N$ there are $S_k,S'_k\in\Omega$ and $v_k\in\R^n$ so that 
$B_k[T]=B_k[S_k]$, $B_k[T']=B_k[S'_k]$, and $B_k[S_k-v_k]=B_k[S'_k-v_k]$. Pick $\tT=(T,\tx)\in\tilde{\Omega}_K$. Let $\tS_k:=(S_k,\tx)\in\tO_K$. There is then $\tS'_k=(S'_k,\ty_k)\in\tO_K$ such that $\pi_2(\tS_k-v_k)=\pi_2(\tS'_k-v_k)$. Let $\tT'_k:=(T',\ty_k)\in\tO_K$. 

\begin{claim}\label{finitely many} There are only finitely many distinct $\ty_k$, $k\in\N$.
\end{claim}

The proof of Claim \ref{finitely many} will follow from the 

\begin{claim}\label{flat} There is $R$ so that if $\tT,\tS\in\tO_K$ and $\bar{d}(\tT,\tS)< 2diam(\Omega)$, then 
$\bar{d}(\tT-v,\tS-v)<R$ for every $v\in\R^n$.
\end{claim}

To prove Claim \ref{flat}, let $\tilde{g}:\tO_K\to\R^{md}$ be a lift of $g$ as constructed above.  Under the identification of $\tO_K$ with $\il (\tf_k)$ used above in the construction of $\tilde{g}$, the point $\tT=(T,\tx)\in\tO_K$ corresponds to $(\tx_0,\tx_1,\ldots)\in\il (\tf_k)$ with: $\tx_0=\tx$; $\tx_1$ such that $\tf_1(\tx_1)=\tx_0$ and $\pi_1(\tx_1)=p^1(\Phi^{-1}(T)); \ldots$ (here $\pi_1:\tX_1\to X_1$ is the covering projection). Then, for $[h]$ in the group $H_1(X;\Z)/K$ of deck translations of $\tO_K$, $\tT+[h]$ (by this we  mean the deck translation $[h]$ applied to $\tT$) corresponds to $(\tx+[h],\ldots,\tx_k+[((f_1)_*\circ\cdots\circ(f_k)_*)^{-1}(h)],\ldots)$ (we have used unimodularity to invert $(f_i)_*$ on $H_1(X;\Z)/K$) . Thus, using equation \ref{g} for $k$ large, we have $$\tilde{g}(T+[h])\approx \tilde{g}_k(\tx_k+[((f_1)_*\circ\cdots\circ(f_k)_*)^{-1}(h)])$$ $$=\tilde{g}_k(\tx_k)+\iota\circ l_k(((f_1)_*\circ\cdots\circ(f_k)_*)^{-1}(h)).$$ Using the easily checked fact that $l_k=l_{k-1}\circ (f_k)_*$, we have $l_k(((f_1)_*\circ\cdots\circ(f_k)_*)^{-1}(h))=l(h)$. Thus $\tilde{g}(\tT+[h])\approx \tilde{g}(\tT)+\iota\circ l(h)$, with an improving approximation as $k\to\infty$. That is, $$\tilde{g}(\tT+[h])= \tilde{g}(\tT)+\iota\circ l(h).$$ Now $[h]\to l(h)$ is an isomorphism of $H_1(X;\Z)/K$ with $GR(\Phi)\subset\R^n$ and $\iota:\R^n\to\R^{md}$ is a linear embedding, so $\iota\circ l: H_1(X;\Z)/K\to\Z^{md}$ is injective. It follows from this and equivariance of $\bar{d}$ that, given $C$, there is an $R$ so that if $\bar{d}(\tT,\tS)>R$ then $|\tilde{g}(\tT)-\tilde{g}(\tS)|>C$.
 
From $g(T-v)=g(T)-\iota(v)$ we deduce $\tilde{g}(\tT-v)=\tilde{g}(\tT)-\iota(v)$. Since $\tilde{g}$ is a lift and $\bar{d}$ is equivariant, there is $C$ so that $\bar{d}(\tT,\tS)<2diam(\Omega)$ implies $|\tilde{g}(\tT)-\tilde{g}(\tS)|<C$. We then have $\bar{d}(\tT,\tS)<2diam(\Omega)\implies |\tilde{g}(\tT)-\tilde{g}(\tS)|<C \implies |(\tilde{g}(\tT)-\iota(v))-(\tilde{g}(\tS)-\iota(v))|<C \implies |\tilde{g}(\tT-v)-\tilde{g}(\tS-v)|<C \implies \bar{d}(\tT-v,\tS-v)<R$, proving Claim \ref{flat}.

Now for the proof of Claim \ref{finitely many}. We have $\bar{d}(\tS_k,\tS'_k)=d(S_k,S'_k)+\td(\tx,\ty)$
and $\bar{d}(\tS_k-v_k,\tS'_k-v_k)=d(S_k-v_k,S'_k-v_k)+0<2diam(\Omega)$. By Claim \ref{flat}, $\td(\tx,\ty)<d(S_k,S'_k)+R$ is bounded. The $\ty_k\in\tX$ all satisfy $\pi_K(\ty_k)=p(T')$ and hence there are $[h_k]\in H_1(X;\Z)$ so that $\ty_k=\ty_0+[h_k]$. As $\td(\ty_k,\ty_0)$ is bounded, there can be only finitely many distinct $[h_k]$ and Claim \ref{finitely many} is established.

Now pick $k_i\to\infty$ with $\ty_{k_i}=:\tx'$ constant and let $\tT':=\tT'_{k_i}$. To ease notation, we assume $k_i=i$, that is, all $\ty_k$ are the same. To conclude the proof of Theorem \ref{Pisot}, we show that $\bar{d}(\tilde{\Phi}^j(\tT),\tilde{\Phi}^j(\tT'))$, $j\in\Z$, is bounded.

First consider $j\ge0$. We have $\pi_2(\tilde{\Phi}^j(\tT))=\pi_2((\Phi^j(T),\tf_K^j(\tx)))=\tf_K^j(\tx)=\pi_2(\tilde{\Phi}^j(\tS_1))$ and similarly, $\pi_2(\tilde{\Phi}^j(\tT'))=\pi_2(\tilde{\Phi}^j(\tS'_1))$.
Also, $\pi_2(\tS_1-v_1)=\pi_2(\tS'_1-v_1)\implies \pi_2(\tilde{\Phi}^j(\tS_1-v_1))=\pi_2(\tilde{\Phi}^j(\tS'_1-v_1))$. That is, $\pi_2(\tilde{\Phi}^j(\tS_1)-\Lambda^jv_1)=\pi_2(\tilde{\Phi}^j(\tS'_1)-\Lambda^jv_1)$, and by Claim \ref{flat} we have $\bar{d}(\tilde{\Phi}^j(\tS_1),\tilde{\Phi}^j(\tS'_1))<R$. Thus $\td(\tf_K^j(\tx),\tf_K^j(\tx'))<R$, so $\bar{d}(\tilde{\Phi}^j(\tT),\tilde{\Phi}^j(\tT'))<R+diam(\Omega)$ for all $j\ge 0$.

For $j<0$ we will need the following:

\begin{claim}\label{balls} Given $R_1$ there is $R_2$ so that if $\tT,\tS\in\tO_K$ satisfy
$\pi_2(\tilde{\Phi}(\tT)-v)=\pi_2(\tilde{\Phi}(\tS)-v)$ for all $v\in B_{R_2}(0)$, then 
$\pi_2(\tT-v)=\pi_2(\tS-v)$ for all $v\in B_{R_1}(0)$.
\end{claim}
 
 Proof of Claim \ref{balls}: Let $R_1$ be given. It is a consequence of invertibility of $\Phi$ that there is an $R_2$ so that if $T,S\in\Omega$ satisfy $B_{R_2}[\Phi(T)]=B_{R_2}[\Phi(S)]$, then 
 $B_{R_1}[T]=B_{R_1}[S]$. (This is ``recognizability".)
 Now let $\tT=(T,\tx)$ and $\tS=(S,\ty)$. Then $\tT-v=(T-v,\tilde{p}^T_{\tx}(v))$ and 
$\tS-v=(S-v,\tilde{p}^S_{\ty}(v))$. Now suppose that $\pi_2(\tilde{\Phi}(\tT)-v)=\pi_2(\tilde{\Phi}(\tS)-v)$ for all $|v|<R_2$. This means that $\tf_K(\tilde{p}^T_{\tx}(\Lambda^{-1}v))=\tf_K(\tilde{p}^S_{\ty}(\Lambda^{-1}v))$, which is to say $\tilde{p}^{\Phi(T)}_{\tf_K(\tx)}(v)=\tilde{p}^{\Phi(S)}_{\tf_K(\ty)}(v)$, for all such $v$. From this we have $\tf_K(\tx)=\tf_K(\ty)$ and $p(\Phi(T)-v)=p(\Phi(S)-v)$ for $|v|<R_2$. So $B_{R_2}[\Phi(T)]=B_{R_2}[\Phi(S)]$ so that $B_{R_1}[T]=B_{R_1}[S]$. In particular,
$p(T)=p(S)$ and $\ty=\tx+[h]$ for some $[h]\in H_1(X;\Z)/K$. Then $\tf_K(\tx)=\tf_K(\ty)=\tf_K(\tx+[h])=\tf_K(\tx)+[f_*(h)] \implies [f_*(h)]=0$. Since $f_*$ is invertible on $H_1(X;\Z)/K$ (unimodularity),
$[h]=0$ and $\tx=\ty$. Together with $B_{R_1}[T]=B_{R_1}[S]$, this last implies that $\tilde{p}^T_{\tx}(v)=\tilde{p}^S_{\ty}(v)$, that is, $\pi_2(\tT-v)=\pi_2(\tS-v)$, for all $v\in B_{R_1}(0)$.

To finish the proof of the theorem, let $R_2=R_2(R_1)$ be as in Claim \ref{balls} and for $m\in\N$ let
$R_2^m:=R_2(R_2(\cdots(R_2(1))\cdots))$, iterated $m$ times. Pick $j<0$ and take $k>R_2^{|j|}$.
Then $\pi_2(\tT-v)=\pi_2(\tS_k)-v$ for $|v|<k$ implies that $\pi_2(\tilde{\Phi}^{-1}(\tT)-v)=\pi_2(\tilde{\Phi}^{-1}(\tS_k)-v)$ for $|v|<R_2^{|j|-1}$, which implies that ..., which gives $ \pi_2(\tilde{\Phi}^j(\tT)-v)=\pi_2(\tilde{\Phi}^j(\tS_k)-v)$ for $|v|<1$. Similarly, for this $k$, we have $\pi_2(\Phi^j(\tS_k-v_k)-v)=\pi_2(\Phi^j(\tS'_k-v_k)-v)$ and $\pi_2(\Phi^j(\tS'_k-v)=\pi_2(\Phi^j(\tT')-v)$ for $|v|<1$. Using Claim \ref{flat} (as in the $j>0$ case above) we conclude that $\bar{d}(\tilde{\Phi}^j(\tT),\tilde{\Phi}^j(\tT'))<diam(\Omega)+R$. Thus $G(T) = G(T'). $

\end{proof}

\section{Connections with the traditional Pisot Substitution Conjecture}\label{finalsection}

The traditional Pisot Substitution Conjecture is as follows:

\begin{conjecture}\label{Pisotconj} If $\Phi$ is a one-dimensional substitution with unimodular
and irreducible incidence matrix $M$ (that is, the characteristic polynomial of $M$ is irreducible over $\Q$) and with inflation $\lambda$ a Pisot number, then the $\R$-action on $\OP$ has pure discrete spectrum.
\end{conjecture}

We will see below that this is a special case of a slight strengthening of Conjecture \ref{homologicalPisot1}.

Given a substitution $\Psi$ with A-P complex $X$ (not necessarily collared), substitution induced map $f:X\to X$, and natural semiconjugacy $p:\Omega_{\Psi}\to\Lim f$, as defined previously, let
$$H^*_f(\Omega_{\Psi}):=p^*(H^*(\Lim f;\Z)).$$
Let us say that a substitution $\Psi$ is {\bf compatible} with a substitution $\Phi$ if there is a homeomorphism $h_{\Psi}:\OP\to\Omega_{\Psi}$ that conjugates some positive powers of the substitution homeomorphisms $\Phi$ and $\Psi$. We define the {\bf essential cohomology} of
$\OP$ to be $$H^*_{ess}(\OP;\Z):=\cap_{\Psi}h_{\Psi}^*(H^*_f(\Omega_{\Psi})),$$
the intersection being over all $\Psi$ compatible with $\Phi$.

\begin{conjecture}\label{hypconjecture2} If $\Phi$ is unimodular and hyperbolic, and $H^1_{hyp}(\Omega_{\Phi};\Z)=H^1_{ess}(\OP;\Z)$, then $G':\OP\to\T^{D'}$ is a.e. 1-to-1.
\end{conjecture}

\begin{conjecture}\label{homologicalPisot2} If $\Phi$ is a unimodular Pisot family substitution and $rank(H^1_{ess}(\Omega_{\Phi};\Z))=D(\Lambda)$ then the $\R^n$-action on $\Omega_{\Phi}$ has pure discrete spectrum.
\end{conjecture}

\begin{example} Consider the one-dimensional substitution $\Phi$ generated by the substitution
on letters: $1\mapsto 12221111$; $2\mapsto 21112$. The incidence matrix is unimodular and irreducible and the inflation $\lambda$  is the fourth power of the golden mean, a Pisot number. Thus $\Phi$ satisfies the hypotheses of the traditional Pisot Substitution Conjecture (and, in fact, the spectrum is pure discrete as the traditional Pisot conjecture is correct if the incidence matrix is $2\times2$ - \cite{HS}). But $D(\Lambda)=deg(\lambda)=2$ while $rank(H^1(\Omega_{\Phi};\Z))=3$, so the hypotheses of Conjecture \ref{homologicalPisot1} are not met. (The cohomology is easily computed by the techniques of \cite{bd1}.) On the other hand, the A-P complex for $\Phi$ is a wedge of two circles; it follows that $H^1_{ess}(\OP;\Z)=H^1_f(\OP)$ has rank two, and the hypotheses of Conjecture \ref{homologicalPisot2} are satisfied.
\end{example}

\begin{Proposition}\label{2impliesPisotconj}  Suppose $\Phi$ is a one-dimensional substitution with unimodular
and irreducible $d\times d$ incidence matrix $M$  and whose inflation $\lambda$ is a Pisot number. Then $rank(H^1_{ess}(\OP;\Z))=d$.
\end{Proposition}

Hence any substitution satisfying the hypotheses of the traditional Pisot Substitution Conjecture also satisfies the hypotheses of Conjecture \ref{homologicalPisot2}.

\begin{lemma}\label{gfactors} Suppose that $\Phi$ is an $n$-dimensional substitution. Let $f:X\to X$ be the substitution induced map on the A-P complex for $\Phi$ and let
$p:\OP\to X$ be the usual map, inducing $\hat{p}:\OP\to \inv f$. Then the map $g$ from $\OP$
onto the maximal equicontinuous factor factors through $\inv f$ via $\hat{p}$.
\end{lemma}
\begin{proof}
Let $f_c:X_c\to X_c$ be the substitution induced map of the collared A-P complex for $\Phi$, let $\pi:X_c\to X$ be the map that forgets collars, and let $p_c:\OP\to X_c$ be as usual. It suffices to show that if $(x_i),(x'_i)\in \inv f_c$ are such that $\hat{\pi}((x_i))=\hat{\pi}((x'_i))$ then the tilings $T=\hat{p_c}^{-1}((x_i))$ and $T'=\hat{p_c}^{-1}((x'_i))$ are regionally proximal. Suppose then that 
$\hat{\pi}((x_i))=\hat{\pi}((x'_i))$. If $x_i$, and hence $x'_i$ are in the interior of an $n$-cell for infinitely many (hence all) $i$, let $i_j\to\infty$, $\rho$, and $v_j$ be such that $B_0[\Phi^{-i_j}(T)]=\rho-v_j=B_0[\Phi^{-i_j}(T')]$. Then the patches $\Phi^{i_j}(\rho)-\Lambda^{i_j}$ are subsets of both $T$ and $T'$. That is, $T$ and $T'$ share patches of arbitrarily large internal diameter, and hence $T$ and $T'$ are proximal. 

If, on the other hand, $x_i$ and $x'_i$ are in the $n-1$ skeleton of $X_c$ for all $i$, $\pi(x_i)=\pi(x'_i)=:y_i$ is in the $n-1$ skeleton of $X$ for all $i$. From the definition of the equivalence relation
that specifies the way in which the prototiles are glued along their boundaries to form $X$, there are, for each $i\in\N$, $T^i_1,\ldots,T^i_{k_i}\in\OP$ and tiles $\tau^i_j\ne\tau^i_{j+1}\in T^i_j$ so that $0\in\tau^i_j\cap\tau^i_{j+1}$ for all $i,j$, $T^i_1=\Phi^{-i}(T)$, and $T^i_{k_i}=\Phi^{-i}(T')$.
The $k_i$ are bounded and we may select $i_l\to\infty$ so that $k=k_{i_l}$ is constant and the collection $\{\tau^{i_l}_j:j=1,\ldots,k\}$ is constant up to translation. Using compactness of $\OP$ we may pass to a subsequence $I_{l_s}\to\infty$ so that $\Phi^{i_{l_s}}(T^{i_{l_s}}_j)$ converges, say to $T_j$, for $j=1,\ldots,k$. As above, $T_j$ and $T_{j+1}$ share patches of arbitrarily large internal diameter, and are hence proximal, for each $j$. As regional proximality is an equivalence relation and contains the proximality relation, $T=T_1$ and $T'=T_k$ are regionally proximal.
\end{proof}

Suppose that $\Phi$ is a one-dimensional (primitive) substitution whose inflation is a degree $d$ 
Pisot unit. Let $X_c$ be the collared A-P complex for $\Phi$, with vertex set $S_c$ and substitution induced map $f:(X_c,S_c)\to (X_c,S_c)$. In an appropriate basis, $f^*:H^1(X_c,S_c;\Z)\to H^1(X_c,S_c;\Z)$ is represented by the transpose, $M^t$, of the incidence matrix for $\Phi$. There is then a unique $f^*$-invariant subgroup $\mathcal{P}$ of $H^1(X_c,S_c;\Z)$ of rank $d$ with the properties that $f^*|_{\mathcal{P}}$ is represented by an integer matrix with eigenvalue $\lambda$, and $\mathcal{P}$ is maximal in the sense that $H^1(X_c,S_c;\Z)/\mathcal{P}$ is torsion free.
By taking the direct limit by $f^*$ of the short exact sequence for the pair
$$0\to\tilde{H}^0(S_c,\Z)\to H^1(X_c,S_c;\Z)\to H^1(X_c;\Z)\to0$$
and noting that all eigenvalues of (a matrix representing) $f^*:\tilde{H}^0(S_c;\Z)\to\tilde{H}^0(S_c;\Z)$ are 0 or roots of unity, we see that $H^1(\OP;\Z)=\dlim f_*:H^1(X_c;\Z)\to H^1(X_c;\Z)$ also contains a unique subgroup, invariant under $\Phi^*$, with the properties of $\mathcal{P}$. We will call this subgroup the {\bf Pisot subgroup} of $H^1(\OP;\Z)$.

\begin{proof}(of Proposition \ref{2impliesPisotconj})
Let $\Phi$ be as in the proposition, let $f:X\to X$ be the substitution induced map on its A-P complex, and let $S$ be the vertex set of $X$. The homomorphism $f^*:H^1(X,S;\Z)\to H^1(X,S;\Z)$ is represented by the transpose $M^t$ of the incidence matrix (in the basis dual to the prototile edges of $X$). As $M$ is invertible over $\Z$, the direct limit of this homomrphism is $\Z^d$. Moreover, the subgroup $\delta^*(H^0(S;\Z))\subset H^1(X,S;\Z)$ is $f^*$-invariant, so, by irreducibility of $M$, this subgroup must be trivial. From the exact sequence for the pair we see
that $f^*:H^1(X,S;\Z)\to H^1(X;\Z)$ is conjugate with $f^*:H^1(X;\Z)\to H^1(X;\Z)$, and hence $H^1(\inv f;\Z)$ is also $\Z^d$. 

Let $X_c$ denote the collared A-P complex for $\Phi$ with vertex set $S_c$ and map $f_c:(X_c,S_c)\to (X_c,S_c)$ and let  $p_c:\OP\to X_c$ and
$p:\OP\to X$ be the standard maps onto the A-P complexes. Let $\pi:(X_c,S_c)\to (X,S)$ be the map that forgets collars. Then $\pi^*:H^1(X,S;\Z)\to H^1(X_c,S_c;\Z)$ is injective and 
it follows that the image of $H^1(X;\Z)$ in $H^1(X_c;\Z)$ under $\pi^*$ is an $f_c^*$-invariant
copy of $\Z^d$ on which $f_c^*$ acts like $M^t$. Thus $H^1_f(\OP)$ has rank $d$ and 
$rank(H^1_{ess}(\OP;\Z))$ is at most $d$.

Suppose that $\Psi$ is a 1-d substitution with inflation $\eta$ such that $\Psi^n$ is conjugate with  
$\Phi^m$, by means of $h_{\Psi}$, for some $n,m\in\N$. It is easily checked that the topological entropies of the homeomorphisms $\Psi$ and $\Phi$ are $\log(\eta)$ and $\log(\lambda)$, hence $\eta^m=\lambda^n$, as conjugacies preserve entropy. Thus $\eta^m$ is a Pisot unit, also of degree $d$
(as $M^n$ is irreducible).
Now, by Lemma \ref{gfactors}, if $f:X\to X$ is the substitution induced map on the A-P  complex for $\Psi$, and $p:\Omega_{\Psi}\to X$ is the usual map, the map $g:\Omega_{\Psi}\to \T^d$ onto the maximal equicontinuous factor factors through
$\inv f$ via $\hat{p}$. Thus $H^1_f(\Omega_{\Psi})=
\hat{p}^*(H^1(\inv f;\Z))\supset g^*(H^1(\T^d;\Z))$. By Theorems \ref{main theorem} and \ref{Pisot},
$g^*$ conjugates $F_A^*:H^1(\T^d;\Z)\to H^1(\T^d;\Z)$ with $\Psi^*$ restricted to the image of $g^*$. The matrix $A$ is a companion matrix for $\eta^m$ (see the proof of Theorem \ref{main theorem} or \cite{BK}). It follows that $H^1_f(\Omega_{\Psi})$ contains the Pisot subgroup for $\Psi^*$ and hence $h_{\Psi}^*(H^1(\Omega_{\Psi}))$ contains the Pisot subgroup for $\Phi$. Thus $rank(H^1_{ess}(\OP;\Z))$ is at least $d$.
\end{proof}

\end{document}